\documentclass[12pt]{amsart}
\usepackage{amsmath,amssymb, xypic,verbatim,amscd,color}
\usepackage{graphicx}
\usepackage{colordvi}

\numberwithin{equation}{section}

\newcommand\frg{\mathfrak{g}}

\newcommand\rk{\operatorname{rk}}

\newcommand\Res{\operatorname{Res}}

\newcommand\mon{\overline{\operatorname{M}}_{0,n}}

\newtheorem{theorem}{Theorem}[section]
\newtheorem{remark}[theorem]{ Remark}

\newtheorem{corollary}[theorem]{Corollary}

\newtheorem{proposition}[theorem]{Proposition}

\newtheorem{lemma}[theorem]{Lemma}

\newtheorem{definition}[theorem]{Definition}

\newtheorem{example}[theorem]{\bf Example}

\begin{document}
\title{Rank-level duality and Conformal Block divisors }

\author{Swarnava Mukhopadhyay}

\address{Department of Mathematics\\ University of Maryland\\ Mathematics Building
\\ College Park, MD 20742-4015}
\email{swarnava@umd.edu}
\begin{abstract} We describe new relations among conformal block divisors in $\operatorname{Pic}(\overline{\operatorname{M}}_{0,n})$. These relations appear from various rank-level dualities of conformal blocks on $\mathbb{P}^1$ with $n$-marked points. We also give a coordinate free description of rank-level duality maps on $\overline{\mathcal{M}}_{g,n}$. 
\end{abstract}

\maketitle
\section{Introduction}Let $\frg$ be a simple Lie algebra, $\mathfrak{h}$ a Cartan subalgebra and $\ell$ a positive integer. Consider an $n$-tuple $\vec{\lambda}\in P_{\ell}(\frg)^n$, where $P_{\ell}(\frg)$ denotes the set of dominant integral weights of level $\ell$ (see Section 2).  Corresponding to this data, we denote the conformal block bundle  on the moduli stack of stable curves of genus $g$ by $\mathbb{V}_{\vec{\lambda}}(\frg,\ell)$. The fibers of conformal block bundles are referred to as conformal blocks and the first Chern classes $c_1(\mathbb{V}_{\vec{\lambda}}(\frg,\ell))$ are known as conformal block divisors.  The ranks $\operatorname{rk}\mathbb{V}_{\vec{\lambda}}(\frg,\ell)$ of conformal block bundles $\mathbb{V}_{\vec{\lambda}}(\frg,\ell)$ are given by the celebrated Verlinde formula.  
We refer the reader to \cite{Sor, TUY, T} for more details. 

Rank-level duality is a duality between conformal blocks on $\mathbb{P}^1$ with $n$-marked points ( with chosen coordinates around the marked points ) associated to two different Lie algebras. T. Nakanishi and A. Tsuchiya proved (see \cite{NT}) that on $\mathbb{P}^1$, certain conformal blocks of $\mathfrak{sl}(r)$ at level $s$ are dual to conformal blocks of $\mathfrak{sl}(s)$ at level $r$. Rank-level duality statements between conformal blocks for symplectic groups was proved by T. Abe \cite{A}. In \cite{M2}, we used diagram automorphisms and T. Abe's result to produce new symplectic rank-level dualities. A rank-level duality for odd orthogonal groups was proved in \cite{M1}. 

In this work, by exploiting the rank-level duality isomorphisms, we explicitly relate conformal block divisors on $\mon$ given by different sets of data (cf. Theorem \ref{main}). By taking advantage of the fact that conformal blocks bundles are globally generated on $\mon$,  we show vanishing of particular conformal block divisors in type A (cf. Corollary \ref{vanishing}). We also give a coordinate free  description of rank-level duality maps on $\overline{\mathcal{M}}_{g,n}$ which enable us to formulate the relations in Theorem \ref{main}. The description of rank-level duality maps on $\overline{\mathcal{M}}_{g,n}$ is non-trivial as it requires twisting by psi classes to make the expression coordinate free.  
Before describing our findings precisely, we put these results into context both in terms of representation theory and with respect to questions about the birational geometry of the moduli space of curves.

Conformal block divisors are base point free, and hence give rise to morphisms on the moduli space $\mon$.  In particular, the conformal block divisors are nef: they non-negatively intersect all curves in $\mon$. From a  representation theoretic perspective, one would like to identify, for a given triple $(\frg, \vec{\lambda}, \ell)$, conditions under which the nonzero-ness of the corresponding class $c_1(\mathbb{V}_{\vec{\lambda}}(\frg,\ell))$ would be guaranteed.  Similarly, one would like to be able to identify different triples which give rise to the same classes.    From the perspective of the divisor theory on $\mon$, such questions of vanishing and identities are motivated by the question whether the cone generated by nef divisors is polyhedral. Many of the conformal block divisors have been shown to be extremal in the nef cone, and questions of finite generation of cones generated by these divisors also remain open.

In \cite{BGM}, it is shown that for $\ell$ above the so-called {\em{critical level}} ({\em{ theta level}} ), conformal block divisors $c_1(\mathbb{V}_{\vec{\lambda}}(\mathfrak{sl}(r),\ell))$ ( for arbitrary $\frg$ ) on $\mon$ are zero. We recover that result in Section \ref{strangevanishing} as a corollary to Theorem \ref{main}.  Numerous identities between conformal block divisors given by different sets of data have been established using various methods by several authors. We refer the reader to \cite{AGSS, AGS, BGM, F, GG} for more details. The identities established in Theorem \ref{main}, and the vanishing of conformal blocks in type A described in Corollary \ref{vanishing} are distinct from earlier results both in terms of the statements
as well as the methods of proof. However, like the earlier results, the existence of
these identities provides  positive support to a question in \cite{F} about the cone of conformal block divisors being polyhedral.

\subsection{Precise Results} After establishing notation and recalling some definitions, we will be able to state the main result of this paper in this section. We first recall the notion of Dynkin index of an embedding.

\subsubsection{Dynkin index}Let $\mathfrak{s}$, $\frg$ be two simple Lie algebras and $\phi : \mathfrak{s} \rightarrow \mathfrak{g}$ an embedding of Lie algebras. Let $(,)_{\mathfrak{s}}$ and $(,)_{\frg}$ denote the normalized Cartan killing forms such that the the length of the longest root is $2$. We define the Dynkin index of $\phi$ to be the unique integer $d_{\phi}$ satisfying 
$$(\phi(x), \phi(y))_{\frg}=d_{\phi}(x,y)_{\mathfrak{s}},$$ for all $x, y \in \mathfrak{s}$. When $\mathfrak{s}=\mathfrak{g}_1\oplus \mathfrak{g}_2$ is semisimple, we define the Dynkin multi-index of $\phi=\phi_1\oplus \phi_2:\mathfrak{g}_1\oplus\mathfrak{g}_2 \rightarrow \mathfrak{g}$ to be $d_{\phi}=(d_{\phi_1}, d_{\phi_2})$. We now recall the definition of conformal embeddings.

\subsubsection{Conformal embedding}\label{confemb}

Let $\phi=(\phi_1,\phi_2): \mathfrak{s}=\mathfrak{g}_1\oplus \mathfrak{g}_2 \rightarrow \frg$ be an embedding of Lie algebras with Dynkin multi-index $\ell=(\ell_1,\ell_2)$. We define $\phi$ to be a conformal embedding $\mathfrak{s}$ in $\mathfrak{g}$  if the following equality holds:
$$\frac{\ell_1\dim{\frg_1}}{g_1^*+\ell_1}+\frac{\ell_2\dim{\frg_2}}{g_2^*+\ell_2}=\frac{\dim{\frg}}{g^*+1},$$ where $g_1$, $g_2$ and $g$ are the dual Coxeter number of the Lie algebras $\frg_1$, $\frg_2$ and $\frg$ respectively. Many familiar and important embeddings are conformal. For a complete list of conformal embeddings we refer the reader to \cite{BB}.

\subsubsection{The main result} Let $\mathfrak{h}^*$ be the dual to a Cartan subalgebra of $\mathfrak{h}$ of $\frg$ and $Q$ be the root lattice of $(\frg,\mathfrak{h})$. For $\Lambda \in \mathfrak{h}^*$, we define a subset $P^{\Lambda}_{\ell}(\frg):=\{\lambda \in P_{\ell}(\frg)| \Lambda-\lambda\in Q\}$. Consider a conformal embedding (see Section \ref{conformal} for properties) $\phi:\frg_1\oplus \frg_2 \rightarrow \frg$, where $\frg_1,\frg_2$ and $\frg$ are simple Lie algebras with Dynkin multi-index $\ell=(\ell_1, \ell_2)$. We extend it to a map of affine Lie algebras $\widehat{\phi}:\widehat{\frg}_1\oplus \widehat{\frg}_2\rightarrow \widehat{\frg}$. Consider a level one integrable highest weight module $\mathcal{H}_{\Lambda}(\frg)$, and restrict it to $\widehat{\frg}_1\oplus \widehat{\frg}_2$. The module $\mathcal{H}_{\Lambda}(\frg)$ decomposes into irreducible integrable $\widehat{\frg}_1\oplus \widehat{\frg}_2$-modules of level $\ell=(\ell_1,\ell_2)$ as follows:$$\bigoplus_{(\lambda,\mu)\in \widetilde{B}(\Lambda)}\widetilde{m}_{\lambda, \mu}^{\Lambda}\mathcal{H}_{\lambda}(\frg_1)\otimes \mathcal{H}_{\mu}(\frg_2)\simeq \mathcal{H}_{\Lambda}(\frg),$$
 where $\widetilde{B}(\Lambda)$ is a finite set parametrizing the components of the decomposition and $\widetilde{m}_{\lambda, \mu}^{\Lambda}$ is the multiplicity of the component $\mathcal{H}_{\lambda}(\frg_1)\otimes \mathcal{H}_{\mu}(\frg_2)$. It is important to point out that for $(\lambda,\mu) \in \widetilde{B}(\Lambda)$, it is necessary that $(\lambda,\mu) \in P^{\Lambda}_{\ell_1}(\frg_1)\times P^{\Lambda}_{\ell_2}(\frg_2)$. We refer the reader to Section \ref{conformal} for more details. For any subset $A=\{a_1,\dots, a_i\}$ of $\{1,\dots, n\}$ and $\vec{\lambda} \in P_{\ell}(\frg)^n$, we denote the $i$-tuple $\{\lambda_{a_1},\dots \lambda_{a_i}\}$ by $\vec{\lambda}_{A}$. 
\begin{remark}
 The following assumptions are motivated by rank-level duality of conformal blocks. All conformal embeddings for which rank-level duality hold are known to satisfy these assumptions. Examples include the embeddings $\mathfrak{sl}(r)\oplus \mathfrak{sl}(s)\rightarrow \mathfrak{sl}(rs)$ and $\mathfrak{sp}(2r)\oplus \mathfrak{sp}(2s)\rightarrow \mathfrak{so}(4rs)$. We refer the reader to \cite{M1} for more details.  
 \end{remark}
Consider $\vec{\Lambda}\in P_1(\frg)^n$, $\vec{\lambda}\in P_{\ell_1}(\frg_1)^n$ and $\vec{\mu}\in P_{\ell_2}(\frg_2)^n$ such that the following holds:
\begin{enumerate}

\item  $\operatorname{rk}\mathbb{V}_{\vec{\Lambda}}(\frg,1)=1$, $\operatorname{rk}\mathbb{V}_{\vec{\lambda}}(\frg_1,\ell_1)=\operatorname{rk} \mathbb{V}_{\vec{\mu}}(\frg_2,\ell_2)$, and $(\lambda_i,\mu_i)\in \widetilde{B}(\Lambda_i)$ for $1\leq i \leq n$.
\item There exists a bijection $f_{\Lambda}$ between $P^{\Lambda}_{\ell_1}(\frg_1)$ and $P^{\Lambda}_{\ell_2}(\frg_2)$ with the property $(\lambda, \mu) \in \widetilde{B}(\Lambda)$, where $\mu=f_{\Lambda}(\lambda)$, $A$ is a subset of $\{1,\dots,n \}$ with $|A|>2$ and $\operatorname{rk}\mathbb{V}_{\vec{\Lambda}_A,\Lambda}(\frg,1)=1$.

\item For every $(\lambda, \mu)\in \widetilde{B}(\Lambda)$, we assume $$\operatorname{rk} \mathbb{V}_{\vec{\lambda}_{A}, \lambda}(\frg_1,\ell_1)=\operatorname{rk} \mathbb{V}_{\vec{\mu}_{A},\mu}(\frg_2,\ell_2) \ \text{and} \ \operatorname{rk} \mathbb{V}_{\vec{\lambda}_{A^c}, \lambda^*}(\frg_1,\ell_1)=\operatorname{rk} \mathbb{V}_{\vec{\mu}_{A^c},\mu^*}(\frg_2,\ell_2),$$ where $\Lambda$ is the unique weight in $P_1(\frg)$ such that $\operatorname{rk}\mathbb{V}_{\vec{\Lambda}_A,\Lambda}(\frg,1)=1$ and $\mu=f_{\Lambda}(\lambda)$.

\end{enumerate}
 The main result of the paper is the following:
\begin{theorem}\label{main}
With the above assumptions, we have the following relation among conformal block divisors in $\operatorname{Pic}(\mon)$:
\begin{eqnarray*}
c_1( \mathbb{V}_{\vec{\lambda}}(\frg_1,\ell_1))+ c_1(\mathbb{V}_{\vec{\mu}}(\frg_2,\ell_2))=\operatorname{rk}\mathbb{V}_{\vec{\lambda}}(\frg_1,\ell_1).\left\{c_1(\mathbb{V}_{\vec{\Lambda}}(\frg,1))+ \sum_{j=1}^n n_{\lambda_j,\mu_j}^{\Lambda_j} \psi_j\right\}\\
\hspace{8cm}- \sum_{i=2}^{\lfloor \frac{n}{2} \rfloor}\epsilon_i\left\{\sum_{\substack{A \subseteq \{1,\dots, n\}\\ |A|=i}} b_{A,A^c}[D_{A,A^c}]\right\},
\end{eqnarray*} where $[D_{A,A^c}]$ denotes the class of the boundary divisor corresponding to the partition $A\cup A^c=\{1,\dots, n\}$, $\epsilon_i=\frac{1}{2}$ if $i=n/2$ and one otherwise, $\psi_j$ is the j-th psi class,  $n^{\Lambda_j}_{\lambda_j,\mu_j}$ and $b_{A,A^c}$ are non-negative integers as defined in Section \ref{conformal} and Section \ref{ranklevel} respectively.
 \end{theorem}
 
 \begin{remark}Since $\operatorname{rk}\mathbb{V}_{\vec{\Lambda}}(\frg,1)=1$, by factorization (see \cite{TUY}) we know that for every partition $A\cup A^c=\{1,2,\dots, n\}$ there exists a unique $\Lambda \in P_1(\frg)$ such that $\operatorname{rk}\mathbb{V}_{\vec{\Lambda}_A,\Lambda}(\frg,1)=1$. The weight $\Lambda$ depends on the partition but for convenience we do not include it in the notation. 
\end{remark}

\begin{remark}
The integer $b_{A,A^c}$ is the order of vanishing of the determinant of the rank-level duality map along the divisor $D_{A,A^c}$ of $\mon$. This is proved in Proposition \ref{geo}.
\end{remark}

 \subsection{Idea of proof}We give two different ways to prove Theorem \ref{main}. The relations described in Theorem \ref{main} were first conjectured and proved using the geometric approach described below. Later we observed that the relations in Theorem \ref{main} suggested an alternate form of Fakhruddin's Chern class formula (see Proposition \ref{rewrite}). We also give a proof of Theorem \ref{main} using the Chern class formula in Proposition \ref{rewrite}. The assumptions in the geometric approach are slightly stronger than in Theorem \ref{main}. We briefly outline both approaches below. 

\subsubsection{Geometric Approach} We give a coordinate free description of rank-level duality maps on the Deligne-Mumford compactification of $\overline{\mathcal{M}}_{g,n}$. The obvious generalization by replacing the sheaf of three conformal blocks (with choice of coordinates) involved in a rank-level duality map by their coordinate free versions does not work since $\widetilde{B}(\Lambda)$ is in general not equal to $B(\Lambda)$ for a level one weight $\Lambda$ of $\widehat{\frg}$ (see Section \ref{conformal}). To incorporate this, we need to twist the rank-level duality map by suitable psi-classes. Our main technique is a coordinate free construction of vertex algebras described in \cite{BF}. We refer the reader to Section \ref{geometry} for more details.

Assuming that rank-level duality holds on $\mathbb{P}^1$ with $n$-marked points and chosen coordinates for the triple $(\vec{\lambda},\vec{\mu},\vec{\Lambda})$, we calculate the order of vanishing $b_{A,A^c}$ of the determinant of the rank-level duality map on a boundary divisor $D_{A,A^c}$. This gives an isomorphism of two line bundles on $\mon$ and Theorem \ref{main} follows by taking Chern classes. We refer the reader to Proposition \ref{geometry1} for a precise result.

\subsubsection{Chern class formula}  The alternate approach to the proof of Theorem \ref{main} is via Fakhruddin's formula for Chern classes of conformal block divisors. We rewrite Fakhruddin's formula with psi classes by using a Lemma in \cite{FG}. Once this is done, the rest of the proof is a direct calculation to show that the left hand side of the relation in Theorem \ref{main} is same as the right hand side. 
\subsection{Acknowledgments}
I thank Prakash Belkale for helpful discussions during the preparation of this manuscript and for pointing out the twisting by psi classes. I would also like to thank Angela Gibney and Han-Bom Moon for useful discussions on the birational geometry of $\mon$ and for detailed comments on an earlier version of this paper. I also thank the anonymous referee for valuable comments. 
 
\section{Notation and some definitions}
Let $\frg$ be a simple Lie algebra over $\mathbb{C}$ and $\mathfrak{h}$ a Cartan subalgebra of $\frg$. The root system associated to $(\frg,\mathfrak{h})$ is denoted by $Q$. The Cartan Killing form $(,)_{\frg}$  on $\frg$ is normalized such that $(\theta, \theta)=2$, where $\theta $ is the longest root. We identify $\mathfrak{h}$ with $\mathfrak{h}^*$ using $(,)_{\frg}$. We define the affine Lie algebra $\widehat{\frg}$ to be 
$$\widehat{\frg}:= \frg\otimes \mathbb{C}((z)) \oplus \mathbb{C}c,$$ where $c$ belongs to the center of $\widehat{\frg}$ and the Lie bracket is given as follows:
$$[X\otimes f(z), Y\otimes g(z)]=[X,Y]\otimes f(z)g(z) + (X,Y)\Res_{z=0}(gdf).c,$$ where $X,Y \in \frg$ and $f(z),g(z) \in \mathbb{C}((z))$. Let $X(n)=X\otimes z^n$ and $X=X(0)$ for any $X \in \frg $ and $n \in \mathbb{Z}$. The finite dimensional Lie algebra $\frg$ can be realized as a subalgebra of $\widehat{\frg}$ under the identification of $X$ with $X(0)$. 

The finite dimensional irreducible modules of $\frg$ are parametrized by the set of dominant integral weights $P_{+} \subset \mathfrak{h}^*$. Let $V_{\lambda}(\frg)$ denote the irreducible module of highest weight $\lambda \in P_{+}$. We fix a positive integer $\ell$ which we call the level. The set of dominant integral weights of level $\ell$ is defined as follows:
$$P_{\ell}(\frg):=\{ \lambda \in P_{+} | (\lambda, \theta) \leq \ell\}.$$
For each $\lambda \in P_{\ell}(\frg)$ there is a unique irreducible integrable highest weight $\widehat{\frg}$-module $\mathcal{H}_{\lambda}(\frg)$ which satisfies the following properties: 
\begin{enumerate}
\item $V_{\lambda}(\frg) \subset \mathcal{H}_{\lambda}(\frg),$
\item The central element $c$ of $\widehat{\frg}$ acts by the scalar $\ell$.
\end{enumerate}

\section{Properties of Conformal Embeddings}\label{conformal}In this section, we recall an important property that characterizes conformal embeddings ( see \ref{confemb} ). First we define the notion of trace anomaly following \cite{KW}.

Let $\frg$ is simple, we define for any level $\ell$ and a dominant weight ${\lambda}$ of level $\ell$, the trace anomaly $\Delta_{\lambda}(\frg,\ell)$ to be the number 
$\frac{({\lambda}, {\lambda}+2{\rho})}{2(g^*+\ell)},$ where $g^*$ is the dual Coxeter number of $\frg$ and ${\rho}$ denotes the half sum of positive roots and $(,)$ is the normalized Cartan killing form. 
If $\mathfrak{g}$ is semisimple, we define the trace anomaly by taking sum of the conformal anomalies over all simple components. 

Let $\phi: \frg_1\oplus \frg_2 \rightarrow \frg$ be an embedding of Lie algebras with Dynkin multi-index $(\ell_1, \ell_2)$. We extend the map $\phi$ to a map of affine Lie algebras in the obvious way:
$$\widehat{\phi}:\widehat{\frg}_1\oplus \widehat{\frg}_2 \rightarrow \widehat{\frg}.$$ 

We consider $\Lambda \in P_{1}(\frg)$ and let $V_{\Lambda}(\frg)$ denote the highest weight irreducible module of the Lie algebra $\frg$. We restrict $V_{\Lambda}(\frg)$ to $\frg_1\oplus \frg_2$. The $\frg$-module $V_{\Lambda}(\frg)$ decomposes into direct sum of $\frg_1\oplus \frg_2$-modules as follows 
$$V_{\Lambda}(\frg) \simeq \bigoplus_{(\lambda, \mu)\in B(\Lambda)}m_{\lambda,\mu}^{\Lambda}V_{\lambda}(\frg_1)\otimes V_{\mu}(\frg_2),$$ where $m_{\lambda,\mu}^{\Lambda}$ is the multiplicity of the component $V_{\lambda}(\frg_1)\otimes V_{\mu}(\frg_2)$ and $B(\Lambda)$ is a finite set. Similarly for $\Lambda \in P_{1}(\frg)$, we consider the highest weight integrable irreducible $\widehat{\frg}$-module $\mathcal{H}_{\Lambda}(\frg)$ and restrict it to $\widehat{\frg}_1\oplus \widehat{\frg}_2$. The module $\mathcal{H}_{\Lambda}(\frg)$ decomposes into $\widehat{\frg}_1\oplus \widehat{\frg}_2$ as follows:
$$\mathcal{H}_{\Lambda}(\frg)\simeq \bigoplus_{(\lambda, \mu)\in \widetilde{B}(\Lambda)}\widetilde{m}_{\lambda, \mu}^{\Lambda}\mathcal{H}_{\lambda}(\frg_1)\otimes \mathcal{H}_{\mu}(\frg_2).$$ Since the integrable modules are infinite dimensional, $|\widetilde{B}(\Lambda)|$ could be infinite. It is easy to see that $B(\Lambda) \subseteq \widetilde{B}(\Lambda)$. In most cases $B(\Lambda)$ is strictly contained in $\widetilde{B}(\Lambda)$. We recall the following from \cite{KW}: 
\begin{enumerate}
\item An embedding is conformal if and only if $\widetilde{B}(\Lambda)$ is finite for all level one weights $\Lambda$. 
\item The action of the Virasoro operators are the same, i.e. for any $n$ the following equality holds:
$$L_n^{\mathfrak{s}}=L_n^{\frg} \in \operatorname{End}(\mathcal{H}_{\Lambda}(\frg)),$$ where the $n$-th Virasoro operators $L_n^{\mathfrak{s}}$ (resp $L_n^{\mathfrak{g}})$ acts at level $\ell$ (resp level one) on the module $\mathcal{H}_{\Lambda}(\frg)$.
\item Let $\frg_1\oplus \frg_2\rightarrow \frg$ be a conformal embedding. If $(\lambda, \mu ) \in \widetilde{B}(\Lambda)$, then $\Delta_{\lambda}({\frg_1}, \ell_1)+\Delta_{\mu}(\frg_2,\ell_2)-\Delta_{\Lambda}(\frg, 1)$ is a non negative integer $n^{\Lambda}_{\lambda,\mu}$. Furthermore the difference of trace anomalies $n^{\Lambda}_{\lambda,\mu}$ is zero if and only if $(\lambda,\mu) \in B(\Lambda)$.
\end{enumerate}
 We only consider conformal embeddings for the rest of this paper. 

\subsection{Examples of Branching Rules}\label{branchingrules} In this section, we write out the branching rules for the conformal embedding $\mathfrak{sl}(2)\oplus \mathfrak{sl}(3)\rightarrow \mathfrak{sl}(6)$. The Dynkin multi-index of the embedding is $(3,2)$. We will use this in Section \ref{explicit} to compute examples of the relations that come from Theorem \ref{main}. The branching rules  below are computed using \cite{ABI}. They are as follows:
\begin{itemize}
\item $\mathcal{H}_{0}(\mathfrak{sl}(6),1) \simeq \mathcal{H}_{0}(\mathfrak{sl}(2),3)\otimes \mathcal{H}_{0}(\mathfrak{sl}(3),2) \oplus \mathcal{H}_{2\omega_1}(\mathfrak{sl}(2),3)\otimes \mathcal{H}_{\omega_1+\omega_2}(\mathfrak{sl}(3),2)$, where $n^{\Lambda}_{\lambda,\mu}$ is $0$ and $1$ respectively.
\item $\mathcal{H}_{\omega_1}(\mathfrak{sl}(6),1) \simeq \mathcal{H}_{\omega_1}(\mathfrak{sl}(2),3)\otimes \mathcal{H}_{\omega_1}(\mathfrak{sl}(3),2) \oplus \mathcal{H}_{3\omega_1}(\mathfrak{sl}(2),3)\otimes \mathcal{H}_{2\omega_2}(\mathfrak{sl}(3),2)$, where $n^{\Lambda}_{\lambda,\mu}$ is $0$ and $1$ respectively.
\item $\mathcal{H}_{\omega_2}(\mathfrak{sl}(6),1) \simeq \mathcal{H}_{2\omega_1}(\mathfrak{sl}(2),3)\otimes \mathcal{H}_{\omega_2}(\mathfrak{sl}(3),2) \oplus \mathcal{H}_{0}(\mathfrak{sl}(2),3)\otimes \mathcal{H}_{2\omega_1}(\mathfrak{sl}(3),2)$, where $n^{\Lambda}_{\lambda,\mu}$ is $0$ for both the components. 
\item $\mathcal{H}_{\omega_3}(\mathfrak{sl}(6),1) \simeq \mathcal{H}_{3\omega_1}(\mathfrak{sl}(2),3)\otimes \mathcal{H}_{0}(\mathfrak{sl}(3),2) \oplus \mathcal{H}_{\omega_1}(\mathfrak{sl}(2),3)\otimes \mathcal{H}_{\omega_1+\omega_2}(\mathfrak{sl}(3),2)$,  where $n^{\Lambda}_{\lambda,\mu}$ is $0$ for both the components. 
\item $\mathcal{H}_{\omega_4}(\mathfrak{sl}(6),1) \simeq \mathcal{H}_{0}(\mathfrak{sl}(2),3)\otimes \mathcal{H}_{2\omega_2}(\mathfrak{sl}(3),2) \oplus \mathcal{H}_{2\omega_1}(\mathfrak{sl}(2),3)\otimes \mathcal{H}_{\omega_1}(\mathfrak{sl}(3),2)$, where $n^{\Lambda}_{\lambda,\mu}$ is $0$ for both the components.
\item $\mathcal{H}_{\omega_5}(\mathfrak{sl}(6),1) \simeq \mathcal{H}_{\omega_1}(\mathfrak{sl}(2),3)\otimes \mathcal{H}_{\omega_2}(\mathfrak{sl}(3),2) \oplus \mathcal{H}_{3\omega_1}(\mathfrak{sl}(2),3)\otimes \mathcal{H}_{2\omega_1}(\mathfrak{sl}(3),2)$, where $n^{\Lambda}_{\lambda,\mu}$ is $0$ and $1$ respectively. 
\end{itemize}

Observe that in all the cases above, the multiplicity $m^{\Lambda}_{\lambda,\mu}$ of a component is always one. This is not true for arbitrary conformal embeddings.

\section{Sheaf of conformal blocks}
In this section we recall the definition of the sheaf of conformal blocks following \cite{TUY}. We also give a coordinate free description of conformal blocks following \cite{Sor} and \cite{T}.
By a $n$-pointed nodal curve, we mean an algebraic curve (with at most nodal singularities) over $\mathbb{C}$ with $n$-distinct marked (smooth) points which satisfies the Deligne-Mumford stability conditions.
A family of $n$-pointed nodal curves is a proper map $\pi :\mathcal{C}\rightarrow \mathcal{B}$ of relative dimension $1$ with sections $s_1,\dots, s_n$ such that the fiber $C_{b}$ with the marked points $s_1(b),\dots, s_n(b)$ is an $n$-pointed nodal curve of arithmetic genus $g$. 

\subsection{Conformal blocks with choice of coordinates}Consider a family $\mathcal{F}=(\pi:\mathcal{C}\rightarrow \mathcal{B};s_1,\dots,s_n, \xi_1,\dots, \xi_n)$ of $n$-pointed nodal curves of genus $g$ with sections $s_i$ and formal coordinates $\xi_i$ around the sections. Consider the $\mathcal{O}_\mathcal{B}$-module of affine Lie algebras  
$$\widehat{\frg}_n(\mathcal{B}):=\frg\otimes_{\mathbb{C}}(\bigoplus_{i=1}^n\mathcal{O}_{\mathcal{B}}((\xi_i))\oplus \mathcal{O}_{\mathcal{B}}.c,$$  where $c$ belongs to the center of $\widehat{\frg}_n(\mathcal{B})$. For $X_i \in \frg$ and $f_i \in \mathcal{O}_{\mathcal{B}}((\xi_i))$, the Lie bracket is defined as follows:
\begin{eqnarray*}
&&[(X_1\otimes f_1,\dots, X_n\otimes f_n), (Y_1\otimes f_1,\dots Y_n\otimes f_n)]\\
 &&\hspace{1cm} =([X_1,Y_1]\otimes f_1g_1,\dots, [X_n,Y_n]\otimes f_ng_n)+ \sum_{i=1}^n(X_i,Y_i)\Res_{\xi_i=0}(g_idf_i)c.
\end{eqnarray*}
We put $\widehat{\frg}(\mathcal{F})=\frg\otimes_{\mathbb{C}}\pi_*(\mathcal{O}_{\mathcal{C}}(*S))$, where $S=\sum_{i=1}^ns_i(\mathcal{B})$. Using the choice of formal coordinates, we may regard $\widehat{\frg}(\mathcal{F}) $ as a Lie subalgebra of $\widehat{\frg}_n(\mathcal{B})$. For $\vec{\lambda}\in P_{\ell}(\frg)^n$, we consider 
$$\mathcal{H}_{\vec{\lambda}}(\mathcal{B}):=\mathcal{O}_{\mathcal{B}}\otimes_{\mathbb{C}} \mathcal{H}_{{\lambda}_1}(\frg)\otimes_{\mathbb{C}}\dots \otimes_{\mathbb{C}} \mathcal{H}_{{\lambda}_n}(\frg).$$
\begin{definition}
The sheaf of covacua $\mathcal{V}_{\vec{\lambda}}(\mathcal{F})$ attached to the family $\mathcal{F}$ is defined to be $\mathcal{H}_{\vec{\lambda}}(\mathcal{B})/\widehat{\frg}(\mathcal{F})\mathcal{H}_{\vec{\lambda}}(\mathcal{B})$.
It is well known that $\mathcal{V}_{\vec{\lambda}}(\mathcal{F})$ is locally free of finite rank and it's dual $\mathcal{V}^{\dagger}_{\vec{\lambda}}(\mathcal{F})$ is known as the sheaf of conformal blocks. 
\end{definition}
\subsection{Coordinate free description}\label{coordinatefreekacmoody} Let $(\pi: \mathcal{C}\rightarrow \mathcal{B}; s_1\dots, s_n)$ be a family of $n$-pointed nodal curves with sections $s_1,\dots, s_n$. Let $S_i=\operatorname{Im}s_i$ and $\mathcal{I}_{S_i}$ be the ideal sheaf of $S_i$. Let $\widehat{\mathcal{O}}_{\mathcal{C}/S_i}$ denote the formal completion of $\mathcal{O}_{\mathcal{C}}$ along $S_i$. Let $K_{\mathcal{C}/S_i}=\lim_{p}\lim_{n}\mathcal{O}_{\mathcal{C}}(pS_i)/\mathcal{I}_{S_i}^{n+1}$, the sheaf of formal meromorphic functions along $S_i$. Let $\widehat{\frg}_{S_i}:=\frg\otimes K_{\mathcal{C}/S_i}\oplus \mathcal{O}_{\mathcal{B}}.c$ with the Lie bracket defined as in Section 2. 

Consider the Lie subalgebra $\widehat{\mathfrak{p}}_{S_i}:=\frg\otimes \widehat{\mathcal{O}}_{\mathcal{C}/S_i}\oplus \mathcal{O}_{\mathcal{B}}.c$ of $\widehat{\frg}_{S_i}$. For $\lambda \in P_{\ell}(\frg)$, let $V_{\lambda}$ be the irreducible finite dimensional $\frg$-module with highest $\lambda$. The $\frg\otimes \mathcal{O}_{\mathcal{B}}$-module structure on $V_{\lambda}$ extends to a $\widehat{\mathfrak{p}}_{S_i}$-module, where $c$ acts by multiplication by $\ell$ and $\frg\otimes \widehat{\mathcal{O}}_{\mathcal{C}/S_i}$ acts by evaluation along $S_i$. Let $M_{\lambda}({\mathcal{C}/S_i})=\operatorname{Ind}_{\widehat{\mathfrak{p}}_{S_i}}^{\widehat{\frg}_{S_i}}V_{\lambda}$ be the $\widehat{\frg}_{S_i}$-Verma module associated to $\lambda$. By PBW theorem, ${M}_{{\lambda}}(C/S_i)$ is isomorphic to $\mathcal{U}(\widehat{\frg}_{S_i})\otimes_{\mathcal{U}(\widehat{\mathfrak{p}}_{S_i})}V_{{\lambda}}$, where $\mathcal{U}({\mathfrak{a}})$ denotes the universal enveloping algebra of a Lie algebra $\mathfrak{a}$. The module $M_{\lambda}(\mathcal{C}/S_i)$ admits an unique irreducible quotient which we denote by $\mathbb{H}_{\lambda}(\mathcal{C}/S_i)$. 

The Lie algebra $\widehat{\frg}_n(\mathcal{C}/\mathcal{B}):=\frg\otimes_{\mathbb{C}}(\bigoplus_{i=1}^n K_{\mathcal{C}/S_i})\oplus \mathcal{O}_{\mathcal{B}}.c$ acts on $\mathbb{H}_{\vec{\lambda}}(\mathcal{C}/\mathcal{B}):=\otimes_{i=1}^n\mathbb{H}_{\lambda_i}(\mathcal{C}/S_i)$, where $\vec{\lambda}=(\lambda_1,\dots,\lambda_n)$ and $\lambda_i\in P_{\ell}(\frg)$. We can identity $\widehat{\frg}(\mathcal{C}/\mathcal{B}):=\frg\otimes_{\mathbb{C}}\pi_*(\mathcal{O}_{\mathcal{C}}(*S))$ as a Lie subalgebra of $\widehat{\frg}_n(\mathcal{C}/\mathcal{B})$.
\begin{definition}
The coordinate free sheaf of covacua $\mathbb{V}_{\vec{\lambda}}(\mathcal{C}/\mathcal{B})$ associated the family $(\pi: \mathcal{C}\rightarrow \mathcal{B}; s_1,\dots, s_n)$ and $\vec{\lambda}\in P_{\ell}(\frg)^n$ is defined to be the following sheaf of coinvariants:
$$\mathbb{H}_{\vec{\lambda}}(\mathcal{C}/\mathcal{B})/\widehat{\frg}(\mathcal{C}/\mathcal{B}). \mathbb{H}_{\vec{\lambda}}(\mathcal{C}/\mathcal{B}).$$
\end{definition}
\begin{remark}
If we choose formal coordinates $\xi_i$ of the sections $s_i$, then $\widehat{\mathcal{O}}_{\mathcal{C}/S_i}\simeq \mathcal{O}_{\mathcal{B}}[[\xi_i]]$ and $K_{\mathcal{C}/S_i}\simeq \mathcal{O}_{\mathcal{B}}((\xi_i))$. Further $\mathbb{H}_{\vec{\lambda}}(\mathcal{C}/\mathcal{B})\simeq \mathcal{H}_{\vec{\lambda}}(\mathcal{B})$ and $\mathbb{V}_{\vec{\lambda}}(\mathcal{C}/\mathcal{B})\simeq \mathcal{V}_{\vec{\lambda}}(\mathcal{F})$.
\end{remark}

\section{The group $\operatorname{Aut}\mathcal{O}$ and a coordinate free construction}

In this section we use a coordinate free description of Vertex algebras given in \cite{BF} and give a coordinate free description of highest weight integrable $\widehat{\frg}$-modules.
 
\subsection{Exponentiating vector fields}We recall a few facts about exponentiating  an action of a vector field on a module following Section 6.3 in \cite{BF}. Let $\mathcal{O}$ denote the complete topological $\mathbb{C}$-algebra $\mathbb{C}[[z]]$ and let $\operatorname{Aut}\mathcal{O}$ be the group of continuous automorphisms of $\mathcal{O}$. Such an automorphism is completely determined by an action on the generator $z$. We can identify $\operatorname{Aut}\mathcal{O}$ with the set of series of the form $a_1z+ a_2z^2+\dots $ with $a_1\in \mathbb{C}^*$, where the group law is given by the usual composition of formal power series. Also consider the subgroup $\operatorname{Aut}_{+}\mathcal{O}$ consisting of elements of the form $z+a_2z^2+\dots$. The following Lemma is easy to prove (cf Lemma 6.2.1 in \cite{BF}):
\begin{lemma}\label{aut}
Let us denote the Lie algebras  $z\mathbb{C}[[z]]\partial_{z}$ by $\operatorname{Der}_{0}\mathcal{O}$ and  $z^2\mathbb{C}[[z]]\partial_{z}$ by $\operatorname{Der}_{+}\mathcal{O}$, then 
\begin{enumerate}
\item $\operatorname{Aut}\mathcal{O}$ is a semi-direct product of the multiplicative group $\mathbb{G}_m$ and $\operatorname{Aut}_{+}\mathcal{O}$.
\item $\operatorname{Aut}_{+}\mathcal{O}$ has the structure of a prounipotent proalgebraic group.
\item $\operatorname{Lie}(\operatorname{Aut}\mathcal{O})=\operatorname{Der}_0\mathcal{O}$, $\operatorname{Lie}(\operatorname{Aut}_{+}\mathcal{O})=\operatorname{Der}_{+}\mathcal{O}$ and the exponential map $exp:\operatorname{Der}_{+}\mathcal{O} \rightarrow \operatorname{Lie}(\operatorname{Aut}_{+}\mathcal{O}$ is an isomorphism. 
\end{enumerate}

\end{lemma}
As pointed out in \cite{BF}, one expects the derivations of a ring to form the Lie algebra of its group automorphisms, but this is not the case here since the ring $\operatorname{Der} \mathcal{O}:=\mathbb{C}[[z]]\partial_z$ is bigger than $\operatorname{Der}_0\mathcal{O}$. So it appears that we may not be able to exponentiate the action of the infinitesimal shift$\partial_z$. The anomaly is resolved by considering the fact that $\operatorname{Aut}\mathcal{O}$ is a semi-direct product of $\mathbb{G}_m$ and $\operatorname{Aut}_{+}\mathcal{O}$.  

Suppose we are given an action of $\operatorname{Lie} \mathbb{G}_m \simeq \mathbb{C}. z\partial_z$ on a vector space $V$. Then this representation can be exponentiated to a representation of the multiplicative group $\mathbb{G}_m$ if an only if, the action of $z\partial_z$ is diagonalizable and its eigenvalues are integers. Then, we can define a $\mathbb{C}^*$ action on $V$ by letting $a \in \mathbb{C}^*$ act by $a^n$ on the eigenvectors of $z\partial_z$ with eigenvalue $n$. 

Next Lemma \ref{aut} tells us that the exponential map from $\operatorname{Der}_{+}\mathcal{O}$ to $\operatorname{Aut}_{+}\mathcal{O}$ is a isomorphism. We can exponentiate a $\operatorname{Der}_{+}\mathcal{O}$ action on a vector space space $V$ to $\operatorname{Aut}_{+}\mathcal{O}$, if it is locally nilpotent, i.e., for any $v\in V$ and $x \in \operatorname{Der}_{+}\mathcal{O}, x^N.v=0$ for $N$ sufficiently large. This indeed guarantees that $\operatorname{exp} x$ is a finite sum. We summarize our discussion as follows:

\begin{proposition}\label{exponential}An action of $\operatorname{Der}_{0}\mathcal{O}$ on a module $V$ can be exponentiated to an action of $\operatorname{Aut}\mathcal{O}$ if the following are satisfied:
\begin{enumerate}
\item The action of $z\partial_z$ is diagonalizable with integral eigenvalues.
\item The action of $\operatorname{Der}_{+}\mathcal{O}$ is locally nilpontent. 
\end{enumerate}

\end{proposition}

\subsection{Virasoro action with integral eigen values}
We recall the Segal-Sugawara action of the Virasoro algebra on integrable highest weight modules $\mathcal{H}_{{\lambda}}(\frg)$. Let us denote $X\otimes z^n$ as $X(n)$, where $z$ is a variable. The normal ordering $:{}:$ is defined by 
\[:X(n)Y(m):=   \left\{
\begin{array}{ll}
      X(n)Y(m)& n<m, \\
      \frac{1}{2}(X(n)Y(m)+Y(m)X(n)) & n=m,\\
      Y(m)X(n) & n>m. \\        
\end{array} 
\right. \]
We will now use the normal ordering $: {}:$ defined above to give the Segal-Sugawara action. Let $\{J^1,\dots, J^{\dim{\frg}}\}$ be an orthonormal basis of $\frg$ with respect to the normalized Cartan Killing form $( ,)_{\frg}$ and $g^*$ denote the dual Coxter number. Then the polynomial vector field $-z^{n+1}\partial_{z}$ acts on $\mathcal{H}_{\lambda}(\frg)$ by the following operator.
$$L_n:=\frac{1}{2(g^*+\ell)} \sum_{m \in \mathbb{Z}}\sum_{a=1}^{\dim \frg}: J^a(m)J^a(n-m):$$

By the above discussion, it follows that the Lie algebra $\operatorname{Der}_0\mathcal{O}$ acts on the $\widehat{\frg}$-module $\mathcal{H}_{\lambda}(\frg)$ by the Segal-Sugawara action. We refer the reader to Section 3.2 in \cite{KW} for more details. Moreover the action of $z^2\mathbb{C}[[z]]\partial_{z}$ is locally nilpotent. The eigenvalues of the operator $L_0=-z\partial_z$ on $\mathcal{H}_{\lambda}(\frg)$ are of the form $\Delta_{\lambda}(\frg,\ell)+i$ which may not integral. Hence the action of $\operatorname{Der}_0 \mathcal{O}$ may not be exponentiated to an action of the group $\operatorname{Aut}\mathcal{O}$. 

We consider the vector space $\mathcal{H}_{\lambda}(\frg)\otimes \mathbb{C} dz^{\Delta_{\lambda}(\frg,\ell)}$. This space is a natural $\operatorname{Der}_0\mathcal{O}$-module where $z^2\mathbb{C}[[z]]\partial_z$ acts locally nilpotently and the eigenvalues of $L_0$ are integral. By Proposition \ref{exponential} (cf. Section 6.3 of \cite{BF}), we get an action of $\operatorname{Aut}\mathcal{O}$ on $\mathcal{H}_{\lambda}(\frg)\otimes \mathbb{C} dz^{\Delta_{\lambda}(\frg,\ell)}$.

\subsection{Torsors and Twists}Let $G$ be a group and $S$ be a non empty set with a simply transitive right action on $G$. Let $V$ be any $G$-module, we can define the $S$-twist of $V$ as the set $$V_S=S\times_{G}V=S\times V/\{(s.g, v) \sim (s,gv)\}.$$ Since the action of $G$ on $S$ is simple transitive, the choice of any point $x \in S$ allows us to identify $G$ with $S$. Further we can identify $V$ with $V_S$ by sending $v \rightarrow (x,v)$. This identification depends on the choice of $x$, however the vector space structure on $V_S$ induced by the above identification is independent of $x$. Any element of $V_S$ can be uniquely written as $(x,v)$, where $v \in V$. We refer the reader to Section 6.4.6 of \cite{BF} for more details. 

Let $C$ be an algebraic curve with at most nodal singularities and $x$ be a smooth point of $C$. Let $\mathcal{O}_x$ be the completed local ring $C$ at a point $x$. Recall that a choice of a formal coordinate at $x$ is same the choice of an isomorphism of $\mathcal{O}_x\simeq \mathbb{C}[[z]]$. Let $\operatorname{Aut}_x$ be the set of all coordinates on $\operatorname{Spec}\mathcal{O}_x$. It is clear that $\operatorname{Aut}_x$ comes with a simply transitive right action of $\operatorname{Aut}\mathcal{O},$ where $\mathcal{O}=\mathbb{C}[[z]]$. Since $\operatorname{Aut}\mathcal{O}$ acts on $\mathcal{H}_{\lambda}(\frg)\otimes \mathbb{C}dz^{\Delta_{\lambda}(\frg,\ell)}$, we can form the following $\operatorname{Aut}_x$-twist:
$$\mathcal{H}_{\lambda}(C/x):=\operatorname{Aut}_x\times_{\operatorname{Aut}\mathcal{O}}(\mathcal{H}_{\lambda}(\frg)\otimes \mathbb{C}dz^{\Delta_{\lambda}(\frg,\ell)}).$$
 Since the action of $\operatorname{Aut}\mathcal{O}$ on $\operatorname{Aut}_x$ is simply transitive, any two formal coordinates of $x$ are relation by an unique element of $\operatorname{Aut}\mathcal{O}$. Hence, by construction there is no preferred choice of coordinates to give a vector space structure on $\mathcal{H}_{{\lambda}}(C/x)$. Moreover if we choose a formal coordinate around the around $x$, we can identify (non-canonically) $\mathcal{H}_{{\lambda}}(C/x)$ with $\mathcal{H}_{{\lambda}}$.  
\subsubsection{Action of $\widehat{\frg}_x$}
Let $\widehat{\frg}_x$ be the coordinate-free affine Kac-Moody algebra constructed in Section \ref{coordinatefreekacmoody}. The vector space $\mathcal{H}_{\lambda}(C/x)$ is naturally a $\widehat{\frg}_x$ module. For completeness, we describe the module structure. As discussed before if $z$ is a formal neighborhood of the point $x$, then any element of $\mathcal{H}_{\lambda}(C/x)$ can be uniquely written as $(z, v)$. Using the formal coordinate, we have a isomorphism between $\widehat{\frg}_x$ and $\widehat{\frg}$. Using the action of $\widehat{\frg}$ on $\mathcal{H}_{\vec{\lambda}}$, we can define an action of $\widehat{\frg}_x$ on $\mathcal{H}_{{\lambda}}(C/x)$. A prori, it appears that the action is dependent on the choice of formal coordinates but it is independent of the choice. This can be shown as follows:

First observe that any automorphism of $h \in \operatorname{Aut}\mathcal{O}$ can be thought of as a power series $z \rightarrow h(z)=a_1z+a_2z^2+\dots$, where $a_1\neq 0$. Given $h$, we can find complex numbers $v_i$'s such that (cf 6.3.1 in \cite{BF})
$$h(z)=\exp\bigg( \sum_{i>0}v_iz^{i+1}\partial_z\bigg)v_0^{z\partial_z}.z,$$ where $v_0^{z\partial_z}.z=v_0z$ and $v_0=a_1$. Let us denote the action of $\operatorname{Aut}(\mathcal{O})$ on $\mathcal{H}_{\lambda}(\frg)\otimes \mathbb{C}dz^{\Delta_{\lambda}(\frg,\ell)}$ by $R$. We know that $(z,v)$ and $(h(z),R(h)^{-1}v)$ denote the same element of $\mathcal{H}_{\lambda}(C/x)$. To show that the action of $\widehat{\frg}_x$ is well defined, it enough to show that as operators: $$R(h)X\otimes f((z)) R(h)^{-1}=X\otimes f(h(z)), \ \mbox{ for all $f(z) \in \mathbb{C}((z))$}.$$
This follows from using Theorem 3.12(1) in \cite{TUY} applied to $\sum_{i>0}v_iz^{i+1}\partial_z$ and the fact that $[z\partial_z,X(m)]=mX(m)$. This completes the description of the action of $\widehat{\frg}_x$ on $\mathcal{H}_{\lambda}(C/x)$.
\begin{lemma}\label{alternate}
Let $C$ be a nodal curve and $x$ be a smooth point of $C$, then we have a canonical identification,
$$\mathbb{H}_{\lambda}(C/x)\simeq \mathcal{H}_{\lambda}(C/x).$$

\end{lemma}
\begin{proof}
We choose a formal coordinate $z$ around the point $x$. This gives an isomorphism of $\widehat{\frg}_x$ with $\widehat{\frg}$ and isomorphism $\phi_1:\mathcal{H}_{\lambda}(C/x)\simeq \mathcal{H}_{\lambda}$ and an isomorphism $\phi_2: \mathbb{H}_{\lambda}(C/x) \simeq \mathcal{H}_{{\lambda}}$. We consider the map $\phi_2{^{-1}}\circ \phi_1$. 

 In Section \ref{coordinatefreekacmoody}, the $\widehat{\frg}_x$-module $\mathbb{H}_{\lambda}(C/x)$ (respectively $\mathcal{H}_{\lambda}(\frg)$) is constructed as a quotient of Verma module $\operatorname{Ind}^{\widehat{\frg}_x}_{\widehat{\mathfrak{p}}_x}V_{\lambda}$ (respectively $\operatorname{Ind}^{\widehat{\frg}}_{\frg \otimes \mathbb{C}[[z]]\oplus \mathbb{C}c}V_{\lambda}$) which is same as $\mathcal{U}(\widehat{\mathfrak{g}}_x)\otimes_{\mathcal{U}(\widehat{\mathfrak{p}}_x)}V_{\lambda}$. Now there are canonical identifications of the form  $\operatorname{Aut}_x\times_{{\operatorname{Aut}}\mathcal{O}} \mathbb{C}[[z]]\simeq \widehat{\mathcal{O}}_{x}$ and $\operatorname{Aut}_x\times_{\operatorname{Aut}\mathcal{O}} \mathbb{C}((z))\simeq{{K}}_{x}$. These give canonical identifications $\widehat{\frg}_x$ with $\operatorname{Aut}_x\times_{\operatorname{Aut}\mathcal{O}}\big(\frg \otimes\mathbb{C}((z))\oplus \mathbb{C}.c\big)$ (similarly for $\widehat{\mathfrak{p}}_x$) and their corresponding universal enveloping algebras. Using this along with the construction of Verma modules tell us that the map $\phi_2^{-1}\circ{\phi_1}$ is independent of the choice of coordinates.

 By definition of the action of $\widehat{\frg}_x$ on $\mathbb{H}_{\lambda}(C/x)$ and $\mathcal{H}_{\lambda}(C/x)$, we see that the isomorphisms satisfy the intertwining property $\phi_i \circ \widehat{\frg}_x=\widehat{\frg}\circ \phi_i$ for $i=\{1,2\}$. Hence the composition $\phi_2^{-1}\circ \phi_1$ is $\widehat{\frg}_x$ equivariant. Since $\phi_2^{-1}\circ\phi_1$ is non-zero, we are done by Schur's Lemma.

\end{proof}

\section{Preliminaries on $\mon$ and the Chern class formula of Fakhruddin} In this section we set up notation and conventions for various important divisors on $\mon$. We also recall the Chern class formula from \cite{F} and rewrite it with psi-classes on $\mon$ using a result of \cite{FG}.

\subsection{Divisors on $\mon$}The moduli space $\mon$ is stratified. A {\em vital codimension $k$ stratum} is an irreducible component of the locus of curves with at least $k$ nodes. The {\em boundary divisors} on $\mon$ are composed of vital codimension $1$-strata labeled by $D_{A,A^c}$, where $A\cup A^c=\{1,\dotsm n\}$ and $|A|,|A^c|\geq 2$. We have the following identification $D_{A,A^c}=D_{A^c,A}$. We denote by $[D_{A,A^c}]$ the linear equivalence class of $D_{A,A^c}$ in $\operatorname{Pic}(\mon)$. 

Let $\overline{\mathcal{M}}_{g,n}$ denote the Deligne-Mumford compactification of the moduli stack of genus $g$ curves with $n$-marked points. The $i$-th psi class $\psi_i$ on $\overline{\mathcal{M}}_{g,n}$ is defined to be the first Chern class of $\mathbb{L}_i$, where $\mathbb{L}_i$ is the line bundle on $\overline{\mathcal{M}}_{g,n}$ whose fiber over the point $(C,p_1\dots,p_n)$ is the cotangent space $T^*_{p_i}(C)$. We recall the following Lemma from \cite{FG}.

\begin{lemma}\label{fg}
The psi classes $\psi_i$ for $1\leq i \leq n$, have the following expression in terms of boundary divisors in $\operatorname{Pic}(\mon)$. 
$$\psi_i=\sum_{\substack{A\subset \{1,2,\dots, n\} \\ 2 \leq |A|\leq n-2 \\ i \in A  }} \frac{(n-|A|)(n-|A|-1)}{(n-1)(n-2)}[D_{A,A^c}].$$
\end{lemma}

\subsection{Chern classes of conformal blocks} N. Fakhruddin's gave the following formula for the first Chern classes of conformal block bundles. We refer the reader to \cite{F} for more details:
\begin{proposition}
Let $\frg$ be a simple Lie algebra, $\ell$ a positive integer and consider an $n$-tuple $\vec{\lambda}=(\lambda_1,\dots, \lambda_n) \in P_{\ell}(\frg)^n$. Then 
\begin{eqnarray*}
&& c_1(\mathbb{V}_{\vec{\lambda}}(\frg,\ell))=\\
&&\ \sum_{i=2}^{\lfloor \frac{n}{2} \rfloor}\epsilon_i  \sum_{\substack{A \subset \{1,2,\dots,n \} \\ |A|=i}}\bigg( \frac{\operatorname{rk}\mathbb{V}_{\vec{\lambda}}(\frg,\ell)}{(n-1)(n-2)} \bigg( (n-i)(n-i-1)\sum_{a\in A}\Delta_{\lambda_a}(\frg, \ell)+ \\
&&\  \ i(i-1)\sum_{a'\in A^c} \Delta_{\lambda_{a'}}(\frg, \ell) \bigg) -\bigg( \sum_{\lambda \in P_{\ell}(\frg)}\Delta_{\lambda}(\frg,\ell).\operatorname{rk}\mathbb{V}_{\vec{\lambda}_A,\lambda}(\frg_1,\ell_1).\operatorname{rk}\mathbb{V}_{{\vec{\lambda}}_{A^c},\lambda^*}(\frg,\ell)\bigg) \bigg) .[D_{A,A^c}],
\end{eqnarray*}
where $[D_{A,A^c}]$ denotes the class of the boundary divisor corresponding to the partition $A\cup A^c=\{1,\dots, n\}$, $\epsilon_i=\frac{1}{2}$ if $i=n/2$ and one otherwise.
\end{proposition}
We rewrite Fakhruddin's formula using psi classes and Lemma \ref{fg}.
\begin{proposition}\label{rewrite}
\begin{eqnarray*}
c_1(\mathbb{V}_{\vec{\lambda}}(\frg,\ell))&=& \operatorname{rk}\mathbb{V}_{\vec{\lambda}}(\frg,\ell)  \bigg( \sum_{j=1}^n \Delta_{\lambda_j}(\frg,\ell)\psi_j\bigg)\\
&& -\sum_{i=2}^{\lfloor \frac{n}{2} \rfloor}\epsilon_i  \sum_{\substack{A \subset \{1,2,\dots,n \} \\ |A|=i}}\left\{\sum_{\lambda \in P_{\ell}(\frg)}\Delta_{\lambda}(\frg,\ell).\operatorname{rk}\mathbb{V}_{\vec{\lambda}_A,\lambda}(\frg,\ell).\operatorname{rk}\mathbb{V}_{{\vec{\lambda}}_{A^c},\lambda^*}(\frg,\ell)\right\}[D_{A,A^c}],
\end{eqnarray*}
 where $\psi_j$ is the $j$-th psi class, $[D_{A,A^c}]$ denotes the class of the boundary divisor $D_{A,A^c}$, $\epsilon_i=\frac{1}{2}$ if $i=n/2$ and one otherwise.
\end{proposition}

\section{Geometric formulation of Theorem \ref{main}} In this section we give a coordinate free description of rank-level duality map on $\overline{\mathcal{M}}_{g,n}$. Using a result of A. Boysal and C. Pauly about compatibility of (see \cite{BP}) factorization with rank-level duality we formulate Theorem \ref{main}. 
\begin{remark}
Though the geometric set up of Theorem \ref{main} works over $\overline{\mathcal{M}}_{g,n}$, but Theorem \ref{main} is only a genus zero result. This is because on positive genus curves the only known rank-level duality for conformal blocks is for the conformal embedding $\mathfrak{g}_2\oplus \mathfrak{f}_4 \rightarrow \mathfrak{e}_8$. In this case, the relations analogous to Theorem \ref{main} has been worked out by the author in \cite{M3}. Strange duality results of \cite{A, BP, Bel1, MO} which holds for curves of arbitrary genus does not fit in the format of Theorem \ref{main} as Theorem \ref{main} is a statement purely about conformal blocks (cf Remark \ref{higher}).
\end{remark}
\subsection{Coordinate free description of rank-level duality}\label{geometry}
Conformal embeddings give rise to maps of conformal blocks associated to stable $n$-pointed curves of genus $g$ with chosen coordinates. These maps are known as rank-level duality maps. We refer the reader to \cite{M1} for more details. We are interested in studying the conformal block divisors on $\mon$, hence we give a coordinate free description of rank-level duality. 

Consider a conformal embedding $\frg_1\oplus \frg_2 \rightarrow \frg$ of Dynkin multi-index $(\ell_1,\ell_2)$. Let $\vec{\Lambda}=(\Lambda_1,\dots, \Lambda_n) \in P_1(\frg)^n$ and  $\vec{\lambda}=(\lambda_1,\dots, \lambda_n)$, $\vec{\mu}=(\mu_1,\dots,\mu_n)$ be such that $(\lambda_i,\mu_i)\in \widetilde{B}(\Lambda_i)$ for $1\leq i \leq n$. We prove the following: 
\begin{proposition}\label{coordinate} There exists a map $\alpha$ between the following vector bundles on $\overline{\mathcal{M}}_{g,n}$ such that after choosing coordinates this map coincides with the rank-level duality map 
$$\alpha: \mathbb{V}_{\vec{\lambda}}(\frg_1,\ell_1)\otimes \mathbb{V}_{\vec{\mu}}(\frg_2,\ell_2)\otimes (\bigotimes_{i=1}^n {\mathbb{L}_i}^{-n_{\lambda_i,\mu_i}^{\Lambda_i}}) \rightarrow \mathbb{V}_{\vec{\Lambda}}(\frg,1),$$ where $n_{\lambda_i,\mu_i}^{\Lambda_i}$'s are difference of trace anomalies as defined in Section \ref{conformal}.
\end{proposition}
\begin{proof} Consider a family of curves $\mathcal{F}=(\pi: \mathcal{C}\rightarrow \mathcal{B};s_1,\dots,s_n;\xi_1,\dots, \xi_n)$ of $n$-pointed nodal curves of genus $g$. Let $s_1,\dots,s_n$ be $n$ sections of $\pi$ with coordinates $\xi_1,\dots, \xi_n$. Since for all $1\leq i\leq n$,  $(\lambda_i,\mu_i)\in \widetilde{B}(\Lambda_i)$, there is a map between 
\begin{eqnarray}\label{branching}
\mathcal{H}_{\vec{\lambda}}\otimes \mathcal{H}_{\vec{\mu}}\otimes \mathcal{O}_{\mathcal{B}}\rightarrow \mathcal{H}_{\vec{\Lambda}}\otimes \mathcal{O}_{\mathcal{B}},
\end{eqnarray}
 where $\mathcal{H}_{\vec{\lambda}}=\mathcal{H}_{\lambda_1}\otimes\dots \otimes  \mathcal{H}_{\lambda_n}$, $\mathcal{H}_{\vec{\mu}}=\mathcal{H}_{\mu_1}\otimes \dots \otimes \mathcal{H}_{\mu_n}$ and $\mathcal{H}_{\vec{\Lambda}}=\mathcal{H}_{\Lambda_1}\otimes \dots \otimes \mathcal{H}_{\Lambda_n}$. This gives rise to a rank-level duality map associated to the family $\mathcal{F}$ between the following locally free sheaves:
$$\mathcal{H}_{\vec{\lambda}}\otimes \mathcal{H}_{\vec{\mu}}\otimes \mathcal{O}_{\mathcal{B}}/ (\frg_1\oplus \frg_2) \otimes_{\mathbb{C}}\pi_*(\mathcal{O}_{\mathcal{C}}(*S)) \ \rightarrow \mathcal{H}_{\vec{\Lambda}}\otimes \mathcal{O}_{\mathcal{B}}/\frg \otimes_{\mathbb{C}} \pi_*(\mathcal{O}_{\mathcal{C}}(*S)).$$ 
Since the embedding is conformal, it follows that the map \ref{branching} commutes with the action $L_n^{\frg_1}+L_n^{\frg_2}$ on the left and on the right by $L_n^{\frg}$ for any integer $n$. Thus conformal embedding and Lemma \ref{alternate} tells us that there is a map of the following vector bundles (formal vector bundles) which is independent of the chosen coordinates $\xi_i$ in the family $\mathcal{F}$. 
\begin{eqnarray}
\mathbb{V}_{\vec{\lambda}}(\mathcal{C}/\mathcal{B},\frg_1, \ell_1)\otimes \mathbb{V}_{\vec{\mu}}(\mathcal{C}/\mathcal{B},\frg_2,\ell_2)\otimes \bigotimes_{i=1}^n (s_i^*(\omega_{\mathcal{C}/\mathcal{B}}))^{-(\Delta_{\lambda_i}(\frg_1,\ell_1)+\Delta_{\mu_i}(\frg_2,\ell_2))}\\
 \rightarrow\mathbb{V}_{\vec{\Lambda}}(\mathcal{C}/\mathcal{B},\frg,1)\otimes \bigotimes_{i=1}^n (s_i^*(\omega_{\mathcal{C}/\mathcal{B}}))^{-\Delta_{\Lambda_i}(\frg,1)}, 
\end{eqnarray}
where $\omega_{\mathcal{C}/\mathcal{B}}$ is the sheaf of relative differential on $\mathcal{C}\rightarrow \mathcal{B}$ and, $\Delta_{\lambda_i}(\frg_1,\ell_1), \Delta_{\mu_i}(\frg_2,\ell_2)$ and $\Delta_{\Lambda_i}(\frg,1)$'s are trace anomalies as defined in Section \ref{conformal}. Since the trace anomalies could be fractional, the line bundles considered in the above map are only formal.  

Since our embedding $\frg_1\oplus \frg_2\rightarrow \frg$ is conformal, we know that for all $1\leq i \leq n$,  the difference of trace anomalies $n^{\Lambda_i}_{\lambda_i,\mu_i}=\Delta_{\lambda_i}(\frg_1,\ell_1)+ \Delta_{\mu_i}(\frg_2,\ell_2)-\Delta_{\Lambda_i}(\frg,1)$ are all integers. Hence associated to the family $(\pi:\mathcal{C}\rightarrow \mathcal{B};s_1, \dots, s_n)$, we have the following map of locally free sheaves:

$$\mathbb{V}_{\vec{\lambda}}(\mathcal{C}/\mathcal{B},\frg_1,\ell_1)\otimes \mathbb{V}_{\vec{\mu}}(\mathcal{C}/\mathcal{B},\frg_2,\ell_2)\otimes \bigotimes_{i=1}^n (s_i^*(\omega_{\mathcal{C}/\mathcal{S}}))^{-n^{\Lambda_i}_{\lambda_i,\mu_i}}\rightarrow \mathbb{V}_{\vec{\Lambda}}(\mathcal{C}/\mathcal{B},\frg,1).$$
This completes the proof.

\end{proof}


\subsection{Factorization and compatiblity}

Let $\mathcal{B}=\operatorname{Spec}\mathbb{C}[[t]]$. We consider a family  $\mathcal{F}: \mathcal{C} \rightarrow \mathcal{B}$ of curves of genus $g$ with $n$-marked points and chosen coordinates such that its special fiber $\mathcal{C}_0$ is a curve over $\mathbb{C}$ with exactly one node and its generic fiber $\mathcal{C}_t$ is a smooth curve. We denote the normalization of $\mathcal{C}_0$ as $\widetilde{\mathcal{C}}_0$. We denote by $\mathcal{V}^{\dagger}_{\vec{\lambda}}(\mathcal{F},\frg,\ell )$ the sheaf of conformal blocks associated to the $\mathcal{F}$, Lie algebra $\frg$ and $\vec{\lambda} \in P_{\ell}(\frg)^n$.

In \cite{TUY}, for every $\lambda \in P_{\ell}(\frg)$, there is a $\mathbb{C}[[t]]$-linear map 
$$s_{\lambda}(t):\mathcal{V}^{\dagger}_{\vec{\lambda},\lambda,\lambda^{*}}({\widetilde{\mathcal{C}}}_0,\frg,\ell)\otimes \mathbb{C}[[t]] \rightarrow \mathcal{V}_{\vec{\lambda}}(\mathcal{F},\frg,\ell)$$ 
such that the following map is a isomorphism 
$$\oplus_{\lambda\in P_{\ell}(\frg)}s_{\lambda}(t):\bigoplus_{\lambda \in P_{\ell}(\frg)}\mathcal{V}^{\dagger}_{\vec{\lambda},\lambda,\lambda^{*}}({\widetilde{\mathcal{C}}}_0,\frg,\ell)\otimes \mathbb{C}[[t]] \rightarrow \mathcal{V}_{\vec{\lambda}}(\mathcal{F},\frg,\ell).$$ This is known as factorization of conformal blocks.

Consider a conformal embedding $\mathfrak{s}\rightarrow \frg$. Assume that all level one highest weight integrable modules of $\widehat{\frg}$ decompose with multiplicity one as $\widehat{\mathfrak{s}}$-modules.  Let $\vec{\Lambda}=(\Lambda_1, \dots, \Lambda_n)$ be an $n$ tuple of level one weights of $\frg$ and $\vec{{\lambda}}\in \widetilde{B}(\vec{\Lambda})$. We get a map $\mathcal{H}_{\vec{{\lambda}}}(\mathfrak{s}) \rightarrow \mathcal{H}_{\vec{\Lambda}}(\mathfrak{g}).$ As discussed in Section \ref{geometry}, we get a $\mathbb{C}[[t]]$-linear map 
 $$\alpha(t) :\mathcal{V}_{\vec{\Lambda}}^{\dagger}(\mathcal{F}, \frg,1) \rightarrow \mathcal{V}_{\vec{\lambda}}^{\dagger}(\mathcal{F}, \mathfrak{s},\ell).$$ For ${\lambda} \in \widetilde{B}(\Lambda)$, we denote by $\alpha_{\Lambda, \lambda}$ the rank-level duality map induced from branching as discussed in Section \ref{geometry}.
 $$\mathcal{V}^{\dagger}_{\Lambda, \Lambda^{\dagger},\vec{\Lambda}}(\widetilde{\mathcal{C}}_0, \frg,1) \rightarrow \mathcal{V}^{\dagger}_{\lambda, \lambda^{\dagger}, \vec{\lambda}}(\widetilde{\mathcal{C}}_0, \mathfrak{s},\ell)$$ and the extension of $\alpha_{\Lambda, \lambda}$ to a $\mathbb{C}[[t]]$-linear map is denoted as follows:
 $$\alpha_{\Lambda, \lambda}(t): \mathcal{V}^{\dagger}_{\Lambda, \Lambda^{\dagger},\vec{\Lambda}}(\widetilde{\mathcal{C}}_0, \frg,1) \otimes \mathbb{C}[[t]]\rightarrow \mathcal{V}^{\dagger}_{\lambda, \lambda^{\dagger}, \vec{\lambda}}(\widetilde{\mathcal{C}}_0, \mathfrak{s},\ell)\otimes \mathbb{C}[[t]].$$
The following proposition from \cite{BP} describes how $\alpha(t)$ decomposes under factorization ( see \cite{TUY} for details on factorization). 
\begin{proposition}\label{keydegen}On $\mathcal{B}$, we have 

$$\alpha(t) \circ s_{\Lambda}(t)=\sum_{{\lambda} \in \widetilde{B}(\Lambda)}t^{n^{\Lambda}_{\lambda}}s_{\lambda}(t)\circ\alpha_{\Lambda,\lambda}(t),$$ where $s_{\Lambda}(t)$, $s_{\lambda}(t)$ are sewing maps associated to factorization and $n^{\Lambda}_{\lambda}$ are positive integers given by the formula $ n^{\Lambda}_{\lambda}=\Delta_{\lambda}(\mathfrak{s},\ell)-\Delta_{\Lambda}(\mathfrak{g},1).$
\end{proposition}

\subsection{Behavior on boundary divisors}\label{ranklevel}We now restrict to the case of genus $0$ curves. Consider a boundary divisor $D_{A,A^c}$ given by the partition $A\cup A^c=\{1,\dots, n\}$. As in Section \ref{geometry}, consider a conformal embedding $\frg_1\oplus \frg_2 \rightarrow \frg$ of Dynkin multi-index $(\ell_1,\ell_2)$. Let $\vec{\Lambda} \in P_1(\frg)^n$ be such that rank of $\mathbb{V}_{\vec{\Lambda}}(\frg,1)$ is one.
Hence by factorization of conformal blocks there exists an unique $\Lambda \in P_{1}(\frg)$ such that the following holds:
$$ \text{rk}\mathbb{V}_{\vec{\Lambda}_A, \Lambda}(\frg,1)=1 \ \hspace{1cm} \ \text{rk}\mathbb{V}_{\vec{\Lambda}_{A^c},\Lambda^{*}}(\frg,1)=1.  $$
Consider $\vec{\lambda}=(\lambda_1,\dots,\lambda_n)$, $\vec{\mu}=(\mu_1,\dots,\mu_n)$ such that $(\lambda_i,\mu_i)\in B(\Lambda_i)$ for all $1\leq i \leq n$. 
\begin{definition}
For a partition $A\cup A^c=\{1,\dots,n\}$, we define $b_{A,A^c}$ to be the following non-negative integer:
\begin{eqnarray*}
b_{A,A^c}=\sum_{(\lambda,\mu)\in \widetilde{B}(\Lambda)}n_{\lambda,\mu}^{\Lambda}.\operatorname{rk} \mathbb{V}_{\vec{\lambda}_{A}, \lambda}(\frg_1,\ell_1).\operatorname{rk} \mathbb{V}_{\vec{\lambda}_{A^c},\lambda^{*}}(\frg_1,\ell_1),
\end{eqnarray*} where $n_{\lambda, \mu}^{\Lambda}$ are the difference of trace anomalies.
\end{definition}
\begin{remark}
The integer $b_{A,A^c}$ depends on the choice of $\Lambda$ for a given boundary divisor $A\cup A^c$ and the weights $\vec{\lambda}$. 

\end{remark}
 Let $\mathcal{B}=\operatorname{Spec}\mathbb{C}[[t]]$. Suppose $\mathcal{V}$ and $\mathcal{W}$ are vector bundles on $\mathcal{B}$ of same rank and let $\mathcal{L}$ be a line bundle on $\mathcal{B}$. Consider a bilinear map $f: \mathcal{V}\otimes \mathcal{W} \rightarrow \mathcal{L}$. Assume that on $\mathcal{B}$, there  are isomorphisms
\begin{eqnarray*}
\oplus {s_i}: \mathcal{V} \rightarrow &  \bigoplus_{i\in I} \mathcal{V}_i \ \  \text{and} \ \ \oplus t_j :  \mathcal{W}  \rightarrow &  \bigoplus_{j\in I}\mathcal{W}_j.
\end{eqnarray*}
Further assume that $\mathcal{V}_i$ and $\mathcal{W}_i$ have the same rank. Let $f_{i,j}$ be maps from $\mathcal{V}_i\otimes \mathcal{W}_j \rightarrow {\mathcal{L}}$ such that $f_{i,j}=0$ for $i\neq j$ and $f=\sum_{i\in I} t^{m_i}(f_{i,i}\circ (s_i\otimes t_i))$. 
\begin{lemma}\label{keygeometry} 
The map $f$ is non-degenerate on $\mathcal{B}^*=\mathcal{B}\setminus \{ t=0\}$ if and only if for all $i\in I$ the maps $f_{i,i}$'s are non-degenerate.

\end{lemma}
We now reveal the geometric significance of the integer $b_{A,A^c}$. Let $\operatorname{rk}\mathbb{V}_{\vec{\Lambda}}(\frg,1)=1$ and we assume that the rank-level duality map is non-degenerate for the tuple $(\vec{\lambda},\vec{\mu}, \vec{\Lambda})$. Further assume that for all $\Lambda \in P_1(\frg)$ such that $\operatorname{rk}\mathbb{V}_{\vec{\Lambda}_A,\Lambda}(\frg,1)=1$, there is an unique bijection between $P^{\Lambda}_{\ell_1}(\frg_1)$ and $P^{\Lambda}_{\ell_2}(\frg_2)$.  Then we have the following proposition:
\begin{proposition}\label{geo}
The integer $b_{A,A^c}$ is the order of vanishing of the determinant of the rank-level duality map on $D_{A,A^c}$. 
\end{proposition}

\begin{proof} Consider the family of genus $0$ curves  $\mathcal{F}:\mathcal{C}\rightarrow \mathcal{B}$ as in Proposition \ref{keydegen} and let $\mathcal{C}_0=C_1\cup C_2$ be a nodal curve with two smooth components meeting at a points. Corresponding to the partition $(A,A^c)$, let the points corresponding to $A$ (resp $A^c$) be on the component $C_1$ (resp $C_2$). The assumptions guarantee that the matrix of the rank-level duality map is block-diagonal. Now the proof follows directly from Proposition \ref{keydegen} and Lemma \ref{keygeometry} applied to the determinant of the rank-level duality map associated to the family $\mathcal{F}$.
\end{proof}

The above discussion can be summarized as follows: 

\begin{proposition}\label{geometry1}
The line bundle  $\det \mathbb{V}_{\vec{\lambda}}(\frg_1,\ell_1)\otimes\det\mathbb{V}_{\vec{\mu}}(\frg_2,\ell_2)$ is isomorphic to the following:
 $$\bigg(\det \mathbb{V}_{\vec{\Lambda}}(\frg,1)\otimes \bigotimes_{i=1}^n\mathbb{L}_i^{\otimes n_{\lambda_i, \mu_i}^{\Lambda_i}}\bigg)^{\otimes{ \operatorname{rk}\mathbb{V}_{\vec{\lambda}}(\frg_1,\ell_1)}}\bigotimes \mathcal{O}_{\mon}\bigg(
 -\sum_{i=2}^{\lfloor \frac{n}{2} \rfloor}\epsilon_i\sum_{\substack{A \subseteq \{1,\dots, n\}\\ |A|=i}} b_{A,A^c}D_{A,A^c}\bigg),$$ where $\epsilon_i=\frac{1}{2}$ if $i=n/2$ and one otherwise. Hence the Chern classes of the line bundles should be equal.
\end{proposition}

\begin{remark}
 Our formulation of Theorem \ref{main} is a consequence of Proposition \ref{geometry1}. The proof of Theorem \ref{main} does not require the stronger assumptions of Proposition \ref{geo}.
\end{remark}

\section{Proof of Theorem \ref{main}}
In this section we give a complete proof of Theorem \ref{main}. For convenience we recall the assumptions on $\vec{\Lambda}$, $\vec{\mu}$ and $\vec{\lambda}$ in Theorem \ref{main}. They are as follows:


\begin{enumerate}

\item  $\operatorname{rk}\mathbb{V}_{\vec{\Lambda}}(\frg,1)=1$, $\operatorname{rk}\mathbb{V}_{\vec{\lambda}}(\frg_1,\ell_1)=\operatorname{rk} \mathbb{V}_{\vec{\mu}}(\frg_2,\ell_2)$, and $(\lambda_i,\mu_i)\in \widetilde{B}(\Lambda_i)$ for $1\leq i \leq n$.
\item There exists a bijection $f_{\Lambda}$ between $P^{\Lambda}_{\ell_1}(\frg_1)$ and $P^{\Lambda}_{\ell_2}(\frg_2)$ with the property $(\lambda, \mu) \in \widetilde{B}(\Lambda)$, where $\mu=f_{\Lambda}(\lambda)$, $A$ is a subset of $\{1,\dots,n \}$ with $|A|>2$ and $\operatorname{rk}\mathbb{V}_{\vec{\Lambda}_A,\Lambda}(\frg,1)=1$.

\item For every $(\lambda, \mu)\in \widetilde{B}(\Lambda)$, we assume $$\operatorname{rk} \mathbb{V}_{\vec{\lambda}_{A}, \lambda}(\frg_1,\ell_1)=\operatorname{rk} \mathbb{V}_{\vec{\mu}_{A},\mu}(\frg_2,\ell_2) \ \text{and} \ \operatorname{rk} \mathbb{V}_{\vec{\lambda}_{A^c}, \lambda^*}(\frg_1,\ell_1)=\operatorname{rk} \mathbb{V}_{\vec{\mu}_{A^c},\mu^*}(\frg_2,\ell_2),$$ where $\Lambda$ is the unique weight in $P_1(\frg)$ such that $\operatorname{rk}\mathbb{V}_{\vec{\Lambda}_A,\Lambda}(\frg,1)=1$ and $\mu=f_{\Lambda}(\lambda)$.

\end{enumerate}

\begin{lemma}\label{calculation0}
For any $\lambda \in P_{\ell_1}(\frg_1) \backslash P^{\Lambda}_{\ell_1}(\frg_1)$, 
$$\operatorname{rk}\mathbb{V}_{\vec{\lambda}_A,\lambda}(\frg_1,\ell_1)=0.$$
\end{lemma}
\begin{proof}
If $\lambda \in P_{\ell_1}(\frg_1)\backslash P^{\Lambda}_{\ell_1}(\frg_1)$, then $\lambda +\sum_{a\in A}\lambda_a $ is not in the root lattice. Hence it does not have any invariants. 
\end{proof}
By the second assumption, we identity $P^{\Lambda}_{\ell_1}(\frg_1)$ and $P^{\Lambda}_{\ell_2}(\frg_2)$ via the given bijection $f_{\Lambda}$.  
We use Proposition \ref{rewrite}, Lemma \ref{calculation0} and write the sum of  $c_1(\mathbb{V}_{\vec{\lambda}}(\frg_1,\ell_1))$ and $c_1(\mathbb{V}_{\vec{\mu}}(\frg_2,\ell_2))$ as follows:

\begin{eqnarray*}
&&\operatorname{rk}\mathbb{V}_{\vec{\lambda}}(\frg_1,\ell_1)  \bigg( \sum_{j=1}^n \big(\Delta_{\lambda_j}(\frg_1,\ell_1)+ \Delta_{\mu_j}(\frg_2,\ell_2)\big)\psi_j\bigg)-\\
&&\ \ \sum_{i=2}^{\lfloor \frac{n}{2} \rfloor}\epsilon_i  \sum_{\substack{A \subset \{1,2,\dots,n \} \\ |A|=i}}\bigg(\sum_{\lambda \in P^{\Lambda}_{\ell_1}(\frg_1)}\big(\Delta_{\lambda}(\frg_1,\ell_1)+ \Delta_{\mu}(\frg_2,\ell_2)\big).\operatorname{rk}\mathbb{V}_{\vec{\lambda}_A,\lambda}(\frg_1,\ell_1).\\
&& \hspace{10cm}\operatorname{rk}\mathbb{V}_{{\vec{\lambda}}_{A^c},\lambda^*}(\frg_1,\ell_1)\bigg).[D_{A,A^c}],
\end{eqnarray*}
where $\epsilon_i=\frac{1}{2}$ if $i=n/2$ and one otherwise, and $\mu$ is of the form $f_{\Lambda}(\lambda)$. 

Since the embedding is conformal, we know that the difference $n^{\Lambda_i}_{\lambda_i,\mu_i}$ of trace anomalies  $\Delta_{\lambda_i}(\frg_1,\ell_1)+ \Delta_{\mu_i}(\frg_2,\ell_2)-\Delta_{\Lambda_i}(\frg,1)$ is a non-negative integer. Thus the above expression can be rewritten as follows:
\begin{eqnarray*}
&& \operatorname{rk}\mathbb{V}_{\vec{\lambda}}(\frg,\ell)  \bigg( \sum_{j=1}^n \big(\Delta_{\Lambda_j}(\frg,1)+ n^{\Lambda_j}_{\lambda_j,\mu_j}\big)\psi_j\bigg)\\
&&\  -\sum_{i=2}^{\lfloor \frac{n}{2} \rfloor}\epsilon_i  \sum_{\substack{A \subset \{1,2,\dots,n \} \\ |A|=i}}\bigg(\sum_{\lambda \in P^{\Lambda}_{\ell_1}(\frg_1)}\big(\Delta_{\Lambda}(\frg,1)+ n^{\Lambda}_{\lambda,\mu}\big).\operatorname{rk}\mathbb{V}_{\vec{\lambda}_A,\lambda}(\frg_1,\ell_1).\operatorname{rk}\mathbb{V}_{{\vec{\lambda}}_{A^c},\lambda^*}(\frg_1,\ell_1)\bigg)[D_{A,A^c}],\\
\end{eqnarray*}
where $\epsilon_i=\frac{1}{2}$ if $i=n/2$ and one otherwise.
The rest of the proof follows from the following lemma: 
\begin{lemma}\label{calculation1}
The following equality holds:
\begin{eqnarray*}
&&\sum_{\substack{ A \subset \{1,\dots,n \}\\ |A|=i}} \sum_{\lambda\in P^{\Lambda}_{\ell_1}(\frg_1)} \Delta_{\Lambda}(\frg,1)\operatorname{rk}\mathbb{V}_{\vec{\lambda}_A,\lambda}(\frg_1,\ell_1).\operatorname{rk}\mathbb{V}_{{\vec{\lambda}}_{A^c},\lambda^*}(\frg_1,\ell_1)\\
&& \quad =\operatorname{rk}\mathbb{V}_{\vec{\lambda}}(\frg_1,\ell_1).\sum_{\substack{ A \subset \{1,\dots,n \}\\ |A|=i}} \sum_{\Lambda\in P_{1}(\frg)} \Delta_{\Lambda}(\frg,1)\operatorname{rk}\mathbb{V}_{\vec{\Lambda}_A,\Lambda}(\frg,1).\operatorname{rk}\mathbb{V}_{{\vec{\Lambda}}_{A^c},\Lambda^*}(\frg,1).
\end{eqnarray*}
\end{lemma}
\begin{proof} By our assumption $\operatorname{rk}\mathbb{V}_{\vec{\Lambda}}(\frg,1)=1$, there exists a unique $\Lambda \in P_1(\frg)$ (depending on the partition) for every partition $A\cup A^c$ of $\{1,\dots,n\}$ such that $\operatorname{rk}\mathbb{V}_{\vec{\Lambda}_A,\Lambda}(\frg,1)=1$. We rewrite the left hand side as follows:
\begin{eqnarray*}
&&\sum_{\substack{ A \subset \{1,\dots,n \}\\ |A|=i}} \sum_{\lambda\in P^{\Lambda}_{\ell_1}(\frg_1)} \Delta_{\Lambda}(\frg,1)\operatorname{rk}\mathbb{V}_{\vec{\lambda}_A,\lambda}(\frg_1,\ell_1).\operatorname{rk}\mathbb{V}_{{\vec{\lambda}}_{A^c},\lambda^*}(\frg_1,\ell_1)\\
&& \quad =\sum_{\substack{ A \subset \{1,\dots,n \}\\ |A|=i}} \Delta_{\Lambda}(\frg,1)\bigg(\sum_{\lambda\in P^{\Lambda}_{\ell_1}(\frg_1)} \operatorname{rk}\mathbb{V}_{\vec{\lambda}_A,\lambda}(\frg_1,\ell_1).\operatorname{rk}\mathbb{V}_{{\vec{\lambda}}_{A^c},\lambda^*}(\frg_1,\ell_1)\bigg)\\
&&  \quad =\operatorname{rk}\mathbb{V}_{\vec{\lambda}}(\frg_1,\ell_1).\sum_{\substack{ A \subset \{1,\dots,n \}\\ |A|=i}} \Delta_{\Lambda}(\frg,1)\\
&&  \quad =\operatorname{rk}\mathbb{V}_{\vec{\lambda}}(\frg_1,\ell_1).\sum_{\substack{ A \subset \{1,\dots,n \}\\ |A|=i}} \sum_{\Lambda\in P_{1}(\frg)} \Delta_{\Lambda}(\frg,1)\operatorname{rk}\mathbb{V}_{\vec{\Lambda}_A,\Lambda}(\frg,1).\operatorname{rk}\mathbb{V}_{{\vec{\Lambda}}_{A^c},\Lambda^*}(\frg,1).
\end{eqnarray*} This completes the proof of the lemma.
\end{proof}

\begin{remark}\label{higher}
There are some strange duality results \cite{A, BP, Bel1, MO} that holds for smooth curves of arbitrary genus. In these cases, the objects involved are conformal blocks on side but some other geometric objects (for example: non-abelian theta functions for vector bundles of fixed degree and rank) on the other side. It is not clear how to define these geometric objects using conformal blocks ( for conformal blocks, we need the Lie algebra to be semisimple). Moreover, it is not clear how these strange dualities  behave over the boundary of the moduli of curves. It will be an interesting and challenging question to study the boundary behavior of these strange dualities and figure out relations analogous to Theorem \ref{main}. 
\end{remark}
\subsection{Examples of the relations in type A}\label{explicit}In this section, we consider rank-level dualities that come from the conformal embedding $\mathfrak{sl}(2)\oplus \mathfrak{sl}(3) \rightarrow \mathfrak{sl}(6)$ over $\overline{\operatorname{M}}_{0,5}$ and $\overline{\operatorname{M}}_{0,6}$ and compute the order of vanishing and the coefficients of $\psi_i$'s in the relations given by Theorem \ref{main}.

\begin{example}
On $\overline{\operatorname{M}}_{0,5}$, let $\vec{\Lambda}=(0,\omega_3,\omega_3,\omega_3,\omega_3)$, $\vec{\lambda}=(2\omega_1, \omega_1,\omega_1,\omega_1,\omega_1)$ and $\vec{\mu}=(\omega_1+\omega_2,\dots,\omega_1+\omega_2)$. This choice of $\vec{\Lambda}$, $\vec{\lambda}$ and $\vec{\mu}$ satisfies the hypothesis of Theorem \ref{main}. Using factorization, we get $\rk \mathbb{V}_{\vec{\lambda}}(\mathfrak{sl}(2),3)=\rk \mathbb{V}_{\vec{\mu}}(\mathfrak{sl}(3),2)=3$. It follows from Example \ref{branchingrules} that $n^{\Lambda_1}_{\lambda_1,\mu_1}=1$ and $n^{\Lambda_i}_{\lambda_i,\mu_i}=0$ for all $i\geq 2$. 
 
The boundary divisors of $\overline{\operatorname{M}}_{0,5}$ are of the form $D_{A,A^c}$, where $A$ is a two element subset of $\{1,\dots,5\}$. We will now compute the order of vanishing $b_{A,A^c}$ along $D_{A,A^c}$.

\begin{itemize}
\item  First, we consider $A=\{1,2\}$. In this case, the unique $\Lambda$ such that $\rk \mathbb{V}_{\vec{\Lambda}_A,\Lambda}(\mathfrak{sl}(6),1)=1$ is $\omega_3$. Now using the calculations in Example \ref{branchingrules}, we get $b_{A,A^c}=0$. By symmetry, it follows that $b_{A,A^c}=0$ for any two element subset $A$ containing $1$. 

\item Now, we let $A=\{2,3\}$. In this case, the unique $\Lambda$ such that $\rk \mathbb{V}_{\vec{\Lambda}_A,\Lambda}(\mathfrak{sl}(6),1)=1$ is $0$. By factorization, we get 
$\rk \mathbb{V}_{(2\omega_1,\omega_1,\dots, \omega_1)}(\mathfrak{sl}(3),2)$ is the sum of $\rk \mathbb{V}_{(\omega_1,\omega_1,2\omega_1)}(\mathfrak{sl}(2),3).\rk \mathbb{V}_{(2\omega_1,2\omega_1,\omega_1,\omega_1)}(\mathfrak{sl}(2),3)$ and $\rk \mathbb{V}_{(\omega_1,\omega_1)}(\mathfrak{sl}(2),3).\rk \mathbb{V}_{(2\omega_1,\omega_1,\omega_1)}(\mathfrak{sl}(2),3)$. Now again by Example \ref{branchingrules} and the definition of $b_{A,A^c}$, we get $b_{A,A^c}$ is $1.\rk \mathbb{V}_{(\omega_1,\omega_1,2\omega_1)}(\mathfrak{sl}(2),3).\rk \mathbb{V}_{(2\omega_1,2\omega_1,\omega_1,\omega_1)}(\mathfrak{sl}(2),3) + 0=2$. Now by symmetry, we get $b_{A,A^c}=2$ for any two element subset $A$ not containing $1$.
\end{itemize}

Further Proposition 1.3 in \cite{BGM} gives us $c_1(\mathbb{V}_{\vec{\lambda}}(\mathfrak{sl}(2),3))=0$. Hence the relation in $\operatorname{Pic}(\overline{\operatorname{M}}_{0,5})$ given by Theorem \ref{main} is the following:
$$c_1(\mathbb{V}_{(\omega_1+\omega_2,\dots,\omega_1+\omega_2)}(\mathfrak{sl}(3),2)=3(c_1(\mathbb{V}_{(0,\omega_3,\omega_3,\omega_3,\omega_3)}(\mathfrak{sl}(6),1))+ \psi_1)- 2\sum_{\substack{ A \subset \{2,3,4,5 \}\\ |A|=2,}}[D_{A,A^c}].$$
\end{example}

\begin{example}
On $\overline{\operatorname{M}}_{0,6}$, let $\vec{\Lambda}=(\omega_1,\dots, \omega_1)$, $\vec{\lambda}=(3\omega_1,\dots,3\omega_1)$ and $\vec{\mu}=(2\omega_2,\dots, 2\omega_2)$. It follows that $\rk \mathbb{V}_{\vec{\Lambda}}(\mathfrak{sl}(6),1)$ is one and all hypothesis of Theorem \ref{main} are satisfied. It follows \cite{M2} that $\rk \mathbb{V}_{\vec{\lambda}}(\mathfrak{sl}(2),3)=\rk \mathbb{V}_{\vec{\mu}}(\mathfrak{sl}(3),2)=1$. Example \ref{branchingrules} tells us that the integer $n^{\Lambda_i}_{\lambda_i,\mu_i}=1$, for all $1\leq i \leq n$. 

Since the conformal blocks in this example are $S_6$ invariant, it enough to compute $b_{A,A^c}$ when $A$ is $\{1,2\}$ and $\{1,2,3\}$. First we consider the case when $A=\{1,2\}$. In this case $\Lambda=\omega_4$, $\lambda=0$ and $\rk \mathbb{V}_{\vec{\lambda}_A,\lambda}=1$ and $\rk \mathbb{V}_{\vec{\lambda}_{A^c},\lambda^*}=1$. Example \ref{branchingrules} tells us $\mu=2\omega_2$ and $n^{\Lambda}_{\lambda,\mu}=0$. Using the definition and the above discussion, we get $b_{A,A^c}=0$. A similar calculation for $A=\{1,2,3\}$ shows that $b_{A,A^c}=0$. Thus by Theorem \ref{main}, the following relation holds in $\operatorname{Pic}(\overline{\operatorname{M}}_{0,6})$:
$$c_1(\mathbb{V}_{\vec{3\omega_1}}(\mathfrak{sl}(2),3))+c_1(\mathbb{V}_{\vec{2\omega_2}}(\mathfrak{sl}(3),2)=c_1(\mathbb{V}_{\vec{\omega}_1}(\mathfrak{sl}(6),1)) + \sum_{i=1}^6 \psi_i.$$
\end{example}

\section{Rank level duality in type A}\label{strangevanishing} In this section we apply our relation on conformal divisors of type $A$ and derive information about ranks of certain conformal blocks. Let $Y_{r,s}$ denote the set of Young diagrams with at most $r$ rows and $s$ columns. For $\lambda=(s\geq \lambda^1 \geq \dots \geq  \lambda^r) \in Y_{r,s}$, one can associate an irreducible $\operatorname{GL}_r$-module $V_{\lambda}$. Two Young diagrams $\lambda_1$ and $\lambda_2$ define the same $\operatorname{SL}_r$ representation if $\lambda_1^i-\lambda_2^i$ is a constant independent of $i$. It is clear that for $\lambda \in Y_{r,s}$, the corresponding dominant integral weight $\lambda \in P_{s}(\mathfrak{sl}(r))$. 

We consider the embedding $\mathfrak{sl}(r)\oplus \mathfrak{sl}(s) \rightarrow  \mathfrak{sl}(rs)$ induced by tensor product $\mathbb{C}^r$ with $\mathbb{C}^s$. It is known that the embedding is conformal with Dynkin multi-index $(s,r)$. The branching rule in \cite{ABI} tells us $(\lambda, \lambda^T) \in B(\omega_{|\lambda|})$, where $|\lambda|$ denotes the number of boxes in the Young diagram of $\lambda$ and $\omega_{|\lambda|}$ denotes the $|\lambda|$-th fundamental weight of $\mathfrak{sl}(rs)$.

Consider $\vec{\lambda} \in Y_{r,s}^n$ such that $\sum_{i=1}^n|\lambda_i|=rs$. Let $\lambda^{T}=(\lambda_1^T,\dots, \lambda_n^T)$ and $\vec{\Lambda}=(\Lambda_1, \dots, \Lambda_n)$, where $\Lambda_i=\omega_{|\lambda_i|}$. It is known that $\operatorname{rk}\mathbb{V}_{\vec{\Lambda}}(\mathfrak{sl}(rs),1)=1$.
\begin{corollary}\label{thetalevel}
With the above notation, we get 
$$c_1(\mathbb{V}_{\vec{\lambda}} (\mathfrak{sl}(r), s))= c_1(\mathbb{V}_{\vec{\lambda}^T}(\mathfrak{sl}(s), r))=0.$$ 
\end{corollary}
\begin{proof} It follows from the rank-level duality  in \cite{NT} that the  above conditions satisfy the axioms of Theorem ~\ref{main}. 
 Further Proposition 5.2 in \cite{F} tells us $c_1(\mathbb{V}_{\vec{\Lambda}}(\mathfrak{sl}(rs),  1)=0$. We also observe that the difference of trace anomaly $n^{\Lambda_i}_{\lambda_i,\lambda_i^T}=0$. Now Theorem \ref{main} tells us the following:
$$c_1(\mathbb{V}_{\vec{\lambda}} (\mathfrak{sl}(r), s))+ c_1(\mathbb{V}_{\vec{\lambda}^T}(\mathfrak{sl}(s), r))+ \mbox{effective sums of boundary divisors}=0.$$ Since conformal blocks bundles are globally generated, this implies that $c_1(\mathbb{V}_{\vec{\lambda}} (\mathfrak{sl}(r), s))+ c_1(\mathbb{V}_{\vec{\lambda}^T}(\mathfrak{sl}(s), r))=0$. But again globally generation forces each of them to be zero.  This completes the proof. 
\end{proof}
\begin{remark}
Corollary \ref{thetalevel} is same as ``above the critical level vanishing" in \cite{BGM}. This was pointed out by an anonymous referee. Further, in the setting of Theorem \ref{main} for arbitrary conformal embeddings $\frg_1\oplus \frg_2 \rightarrow \frg$, if we choose weights such that $c_1(\mathbb{V}_{\vec{\Lambda}}(\frg,1))=0$ and $n^{\Lambda_i}_{\lambda_i,\mu_i}=0$ for all $1\leq i \leq n$, we get vanishing results analogous to Corollary \ref{thetalevel}. The only known non-trivial examples ( follows from Section 5 in \cite{F}) satisfying the hypothesis mentioned above are the ones in Corollary \ref{thetalevel}. It will be an interesting question to find more examples of rank-level duality that satisfies the hypothesis discussed here.
\end{remark}
 Let $D_{A,A^c}$ be a boundary divisor given by the partition $A\cup A^c$ and $\Lambda \in P_1(\mathfrak{sl}(rs))$  be such that $\operatorname{rk}\mathbb{V}_{\vec{\Lambda}_A,\Lambda}(\mathfrak{sl}(rs),1)=1$.  With the above notation, we have the following corollary:
\begin{corollary}\label{vanishing} If $\lambda \in \widetilde{B}(\Lambda)\backslash B(\Lambda)$, then $\mathbb{V}_{\vec{\lambda}_A, \lambda}(\mathfrak{sl}(r),s) \otimes\mathbb{V}_{\vec{\lambda}_{A^c},\lambda^*}(\mathfrak{sl}(r),s)$ is zero. Hence 
$$c_1(\mathbb{V}_{\vec{\lambda}_A, \lambda}(\mathfrak{sl}(r),s))=c_1(\mathbb{V}_{\vec{\lambda}_{A^c},\lambda^*}(\mathfrak{sl}(r),s))=0.$$
\end{corollary}

\begin{remark} The ``theta-level" defined in \cite{BGM} is at least one if all the weights $\lambda_i$'s of $\mathbb{V}_{\vec{\lambda}}(\frg,\ell)$ are non trivial. Hence repeating the strategy of Corollary \ref{thetalevel} do not yield any new vanishing result. The vanishing statement in Corollary \ref{vanishing} is new  and is different from known vanishing results about conformal blocks in the existing literature.
\end{remark}

\bibliographystyle{plain}

\begin{thebibliography}{10}
\bibitem{ABI}D. Altschuer, M. Bauer, C. Itzykson, {\em The branching rules of conformal embeddings}, Comm. Math. Phys. {\bf 132} (1990), no. 2, 349-364.
\bibitem{AGSS} M. Arap, A. Gibney, J. Stankewicz, D. Swirnarski, {\em $sl_n$ level 1 conformal blocks divisors on $\mon$}, Int. Math. Res. Not. IMRN {2012}, Art. ID rnr064, 47 pp.
\bibitem{AGS} V. Alexeev, A. Gibney, D. Swirnarski, {\em Conformal blocks divisors on $\mon$ from $sl_2$}, Proceedings of 
\bibitem {A} T. Abe, {\em Strange duality for parabolic symplectic bundles on a pointed projective line}, Int. Math. Res. Not. IMRN { 2008}, Art. ID rnr121, 47 pp.

\bibitem {BGM} P. Belkale, A. Gibney, S. Mukhopadhyay, {\em Vanishing and identities of conformal blocks divisors}, Final version in Algebr. Geom. 2 (2015), no. 1, 62–90.
\bibitem {BB} A. Bais, P. Bouwknegt, {\em A classification of subgroup truncations of the bosonic string}, Nuclear Phys. B {\bf 279}(1987), no. 3-4, 561-70.
\bibitem {BF} E. Frenkel, D. Ben-Zvi, {\em Vertex Algebras on algebraic curves}, Mathematical Surveys and Monographs {\bf 88}, AMS 2001.
\bibitem {BP}  A. Boysal, C. Pauly, {\em Strange duality for Verlinde spaces of exceptional groups
at level one}, Int. Math. Res. Not. 2009; doi:10.1093/imrn/rnp151.
\bibitem{Bel1} P. Belkale, {\em The strange duality conjecture for generic curves}, J. Amer. Math. Soc. 21 (2008), no. 1, 235–258 (electronic).
\bibitem{FG} G. Farkas, A. Gibney, {\em Mori cones of moduli spaces of pointed curves of small genus }, Trans. Amer. Math. Soc. {\bf 355} (2003), no. 3, 1183-1199 (electronic), DOI 10.1090/S0002-9947-02-03165-3.
\bibitem{F}  N. Fakhruddin,{\em Chern classes of conformal blocks,} Contemp. Math., {\bf 564}, Amer. Math. Soc., Providence, RI, 2012, 145-176.
\bibitem{GG} A. Gibney, N. Giansiracusa, {\em The cone of type A, level one conformal blocks divisors}, Adv Math, (2012), Volume 231.
\bibitem {KW} V. Kac, M. Wakimoto, {\em Modular and conformal invariant constraints in representation theory of affine algebras}, Advances in Mathematics {\bf 70} (1988):156-234.
\bibitem{MO} A. Marian, D. Oprea,{\em The rank-level duality for non-abelian theta functions}, Invent. Math. 168(2007), no. 2, 225-247.
\bibitem {M1} S. Mukhopadhyay, {\em Rank-level duality of conformal blocks for odd orthogonal Lie algebras in genus $0$}, To appear in Transactions of the AMS.
\bibitem {M2}S. Mukhopadhyay,  {\em Diagram automorphisms to rank-level duality}, arXiv:1308:1756v1.
\bibitem{M3} S. Mukhopadhyay, {\em Strange duality of Verlinde spaces for $G_2$ and $F_4$}, 	arXiv:1504.03757.
\bibitem {NT} T. Nakanishi, A. Tsuchiya, {\em Level-rank duality of WZW models in conformal field theory}, Comm. Math. Phys. Volume {\bf 144} (1992), no. 2, 351-372.
\bibitem {Sor} C. Sorger, {\em Le formula de Verlinde}, Ast{$\acute{e}$}rique, (1996), pp. Exp. No. 794, 3, 87-114. S$\acute{e}$minaire Bourbaki, Vol. 1994/95.
\bibitem {TUY} A. Tsuchiya, K. Ueno, Y. Yamada, {\em Conformal field theory on universal family of stable curves with
gauge symmetries}, Integrable systems in quantum field theory and statistical mechanics, 459-566,
Adv. Stud. Pure Math. {\bf  19}, Academic Press, Boston, MA, 1989.
\bibitem {T} Y. Tsuchimoto, {On the coordinate-free description of the conformal blocks}, J. Math. Kyoto Univ. {\bf 33} (1993) 29-49.





\end{thebibliography}
\def\noopsort#1{}

\end{document}